\documentclass[12pt,a4paper,reqno]{amsart} 
\usepackage{amssymb}
\usepackage{latexsym}
\usepackage{amsmath}
\usepackage{mathrsfs}
\usepackage{amsthm}
\usepackage{xcolor,soul}
\usepackage{calc}                         
\usepackage[english]{babel}
\usepackage{cite}

\usepackage{xspace}

\newcommand{\bs}[1]{\boldsymbol{#1}}
\newcommand{\abs}[1]{\left|#1\right|}
\newcommand{\prn}[1]{\left(#1\right)}

\newcommand{\scal}[2]{\langle #1,#2\rangle}

\newcommand{\re}{\mathbf R}
\newcommand{\rr}[1]{\mathbf R^{#1}}

\newcommand{\nn}[1]{\mathbf N^{#1}}

\newcommand{\nm}[2]{\Vert #1\Vert _{#2}}

\newcommand{\op}{\operatorname{Op}}

\newcommand{\sets}[2]{\{ \, #1\, ;\, #2\, \} }
\newcommand{\ep}{\varepsilon}

\newcommand{\cdo}{\, \cdot \, }

\newcommand{\vrum}{\vspace{0.1cm}}

\newcommand{\maclS}{\mathcal S}

\newcommand{\mascF}{\mathscr F}

\newcommand{\mascP}{\mathscr P}
\newcommand{\mascS}{\mathscr S}

\numberwithin{equation}{section}          

\newtheorem{thm}{Theorem}
\numberwithin{thm}{section}

\newcommand{\rubrik}{}
\newtheorem{prop}[thm]{Proposition}
\newtheorem{cor}[thm]{Corollary}
\newtheorem{lemma}[thm]{Lemma}

\theoremstyle{definition}

\newtheorem{defn}[thm]{Definition}

\theoremstyle{remark}
\newtheorem{rem}[thm]{Remark}              

\author{Ahmed Abdeljawad}

\address{Department of Mathematics,
	University of Turin, Italy}

\email{ahmed.abdeljawad@unito.it}

\author{Sandro Coriasco}

\address{Department of Mathematics,
	University of Turin, Italy}

\email{sandro.coriasco@unito.it}

\author{Nenad Teofanov}

\address{Department of Mathematics and Informatics,
	University of Novi Sad, Serbia}

\email{nenad.teofanov@dmi.uns.ac.rs}

\title{\textbf{Bilinear pseudo-differential operators
		with Gevrey-H{\"o}rmander symbols}}

\begin{document}
	


\begin{abstract}
We consider bilinear pseudo-differential operators who\-se symbols posses Gevrey type regularity
and may have a sub-ex\-po\-nen\-tial growth at infinity, together with all
their derivatives. It is proved that those symbol classes can be described by
the means of the short-time Fourier transform and modulation spaces.
Our first main result is the invariance property of the corresponding bilinear operators.
Furthermore we prove the continuity of such operators when acting on modulation spaces.
As a consequence, we derive their continuity on anisotropic Gelfand-Shilov type spaces.
We consider both Beurling and Roumieu type symbol classes and Gelfand-Shilov spaces.
\end{abstract}

\maketitle

\section*{Introduction}\label{sec0}

\par

The study of multilinear operators has been influenced by the Calde\-r\'on-Zygmund theory. Indeed, one of the achievements of Coifman-Meyer's pioneer work \cite{CoMe}
is the realization of pseudo-differential operators in terms of singular integrals of Calder\'on-Zygmund type. Their approach is based on a multilinear point of view and have had a far reaching impact in operator theory and partial differential equations. For example, boundedness of a class of translation invariant bilinear operators on Lebesgue spaces is proved in  \cite{CoMe}. Furthermore, the bilinear Calder\'on-Zygmund theory
developed by Grafakos and Torres \cite{GT} paved the way to the extension of those results to  bilinear pseudo-differential operators which are non-translation invariant, i.e. whose
symbols may depend on the space variable as well. 
We refer to \cite{BO} for a brief survey and  discussion of applications to  partial differential equations, and to \cite{BeMaNaTo} for a systematic study of  bilinear pseudo-differential operators with symbols in bilinear
H\"ormander classes. See also \cite{KT} for a recent  contribution in the context of Triebel-Lizorkin and local Hardy spaces.

Another type of results concerns bilinear (and multilinear) operators whose symbols are not necessarily smooth.
Their continuity properties on modulation spaces were first observed in \cite{BKasso2004}.
In contrast to  classical bilinear pseudo-differential operators considered in e.g. \cite{CoMe},
these operators are treated by the techniques of time-frequency analysis,
see also \cite{BKasso2006, BeGrHeOk, BeMaNaTo, CKasso, MOPf}.

In this paper, we employ the techniques of time-frequency analysis and modulation spaces, and
consider bilinear pseudo-differential operators
of Gevrey-H{\"o}rmander type whose symbols are of infinite order and may have a (super-)exponential growth at infinity, together with all
their derivatives. The linear counterpart of such operators is considered in \cite{CaTo}, within the environment of isotropic Gelfand-Shilov spaces of functions and distributions, see also
\cite{AbCoTo, To25}, and extended to the anisotropic setting in \cite{AbCaTo, AbTo}.
The main purpose of this paper is to extend some boundedness results given there to the bilinear case.

More precisely, we consider Gevrey-H{\"o}rmander type symbols $a \in
{\Gamma _{(\omega)}^{\sigma,\mathbf{s}}} (\rr {3d})$ (or $ a\in {\Gamma _{(\omega)}^{\sigma,\mathbf{s},0}} (\rr {3d})$),
see Definition \ref{Def:GammaSymb2}, and the corresponding  pseudo-differential  operators, denoted by
$\op_{r, t}(a)$, see \eqref{eq:BilPdo} below. When $r=t=0$ we recover the Kohn-Nirenberg correspondence considered, e.g., in \cite{BKasso2004,BeGrHeOk}, while for $r=t=1/2$ we obtain the Weyl correspondence, considered in \cite{Te5,Te7}.
We remark that, in view of this choice of symbol classes, we cannot rely on arguments based on standard
(e.g., Littlewood-Paley) localization techniques. The substitute for this is, from the very beginning, a ``global approach'',
aimed at obtaining and employing appropriate characterizations of the involved objects,
in terms of suitable estimates which hold true on the whole Euclidean spaces.

The paper is organized as follows. In Section \ref{sec1} we collect necessary definitions, background material and basic facts
on Gelfand-Shilov spaces, weight functions, modulation spaces, symbol classes and the corresponding bilinear pseudo-differential operators. In Section 
\ref{sec2}, we first study exponential-type operators on Gelfand-Shilov space 
and prove the corresponding invariance properties.
We proceed with a characterization of the symbol spaces in terms of their regularity and decay properties,
and suitable estimates related to modulation spaces. Finally, we prove  our main
results about the continuity of bilinear operators in Section \ref{sec3}. The proofs of a few technical results are collected in the Appendix.

\stepcounter{section}
\section{Preliminaries}\label{sec1}

\par

In this section we provide notation and background material which will be used throughout the paper. Proofs and details are in general omitted, since they can be found, e.g., in \cite{CoPiRoTe05,Fe2,Fe3,Fe4,FG1,FG2,FG4,Gro,Te5,To2,To7,To8}.

\par

We use the standard notation for Euclidean spaces and multiindeces, cf. \cite{Ho1}. For example,
if $x=(x_{1},...,x_{d})$ $\in\rr d$ and
$\alpha=(\alpha_{1},...,\alpha _{d})\in\nn d$, then $x^{\alpha}=x_{1}^{\alpha_{1}}\dots x_{d}^{\alpha_{d}}$,
$\partial_{x}^{\alpha}=\partial_{x_{1}}^{\alpha_{1}} \dots\partial_{x_{d}}^{\alpha_{d}}$,
$|\alpha|=\alpha_{1}+\dots+\alpha_{d}$ and $\alpha!=\alpha_{1}!\dots\alpha_{d}!$. Here $ \mathbf{N} $ denotes the set of non-negative integers.
If $\alpha\in\nn d$, then $\alpha>0$ means that $\alpha_{j} > 0$ for every $j=1,\dots, d,$ and similarly for
$\alpha \geq 0$. We write $A(\theta )\lesssim B(\theta )$, $\theta \in \Omega$,
if there is a constant $c>0$ such that $|A(\theta )|\le c|B(\theta )|$
for all $\theta \in \Omega$. We write $A(\theta)\asymp B(\theta)$
if $A(\theta )\lesssim B(\theta )$ and $B(\theta )\lesssim A(\theta )$
for all $\theta \in \Omega$. Here $\Omega$ is an open subset of $\rr N$.
If $\mathcal{B}_1$ and $\mathcal{B}_2$ are  topological vector spaces, then $ \mathcal{B}_1 \hookrightarrow \mathcal{B}_2$
means that $\mathcal{B}_1$ is  continuously embedded into $\mathcal{B}_2$.
By $\mathscr S (\rr d)$ we denote the Schwartz space of rapidly decreasing functions, and  $\mathscr S'(\rr d)$ denotes its dual space of tempered distributions.

\par

\subsection{Gelfand-Shilov spaces} \label{subsec1.2}

Let $h,s,\sigma > 0$ be fixed. Then $\maclS _{s;h}^{\sigma}(\rr d)$
is the Banach space of all $f\in C^\infty (\rr d)$ such that
\begin{equation}\label{gfseminorm}
\nm f{\maclS _{s;h}^{\sigma}}\equiv \sup_{\alpha ,\beta \in
	\mathbf N^d} \sup_{x \in \rr d} \frac {|x^\alpha \partial ^\beta
	f(x)|}{h^{|\alpha  + \beta |} \alpha !^s\, \beta !^\sigma}<\infty,
\end{equation}
endowed with the norm \eqref{gfseminorm}.
Obviously,  $\maclS _{s;h}^{\sigma}(\rr d)$ increases as $h$, $s$  and $\sigma$
increase, and it is contained in $\mathscr S (\rr d)$ for every $h,s,\sigma > 0$.

\par

The {Gelfand-Shilov spaces} $\maclS _{s}^{\sigma}(\rr d)$ and
$\Sigma _{s}^{\sigma}(\rr d)$ are defined as the inductive and projective
limits respectively of $\maclS _{s;h}^{\sigma}(\rr d)$, i.e.
\begin{equation}\label{GSspacecond1}
\maclS _{s}^{\sigma}(\rr d) = \bigcup _{h>0}\maclS _{s;h}^{\sigma}(\rr d)
\quad \text{and}\quad \Sigma _{s}^{\sigma}(\rr d) =\bigcap _{h>0}
\maclS _{s;h}^{\sigma}(\rr d),
\end{equation}
with the usual inductive and projective limit topologies. Note that
$\Sigma _{s}^{\sigma}(\rr d)\neq \{ 0\}$, if and only if
$s+\sigma \ge 1$ and $(s,\sigma )\neq
(\frac 12,\frac 12)$, and
$\maclS _{s}^{\sigma}(\rr d)\neq \{ 0\}$, if and only
if $s+\sigma \ge 1$, see \cite{GS, Pi88}.
For every $s,\sigma >0$ we have
\begin{equation}\label{Eq:GSEmbeddings}
\Sigma _s^\sigma (\rr d)
\hookrightarrow
\maclS _s^\sigma (\rr d)
\hookrightarrow
\Sigma _{s+\ep}^{\sigma +\ep}(\rr d)
\hookrightarrow
\mascS (\rr d)
\end{equation}
for every $\ep >0$. If $s+\sigma \ge 1$, then
the last two inclusions in \eqref{Eq:GSEmbeddings} are dense,
and if in addition $(s,\sigma )\neq (\frac 12,\frac 12)$
then the
first inclusion in \eqref{Eq:GSEmbeddings} is dense.
Moreover, for $\sigma<1$ the elements of $\maclS _{s}^{\sigma}(\rr d)$ admit entire
extensions to $\mathbf{C}^d$ satisfying suitable exponential bounds, \cite{GS}.

\par

\begin{rem}
Note that the original definition in \cite{GS} is given with   $s,\sigma \ge 0$. Then
$\maclS _{0}^{\sigma}(\rr d)\neq \{ 0\}$ ($\Sigma _{0}^{\sigma}(\rr d)\neq \{ 0\}$) and $\maclS _{s}^{0}(\rr d)\neq \{ 0\}$
($\Sigma _{s}^{0}(\rr d)\neq \{ 0\}$), if and only if
$\sigma > 1$ and $s>1$ respectively, see \cite[Chapter IV 8.1]{GS}.
In fact, $\maclS  _{0}^{\sigma}(\rr d)\neq \{ 0\}$ consists of compactly supported Gevrey functions, while
$\maclS _{s}^{0}(\rr d)\neq \{ 0\}$ contains functions whose Fourier transforms are compactly  supported Gevrey functions.
\end{rem}

\par

The spaces $\maclS _{s}^{\sigma}(\rr d)$ and $\Sigma _{s}^{\sigma}(\rr d)$
combine  global regularity with suitable decay properties at infinity, thus offering an abstract  functional
analytic framework for some problems in mathematical physics, \cite{Gram, NR}.
The following result is a well-known characterization of $\maclS _{s}^{\sigma}(\rr d)$ and $\Sigma _{s}^{\sigma}(\rr d)$
in terms of the exponential decay of derivatives of their elements.
Although the proof is standard, it contains some tools relevant for our investigations, and is therefore included in Appendix \ref{appendix}.

\begin{lemma} \label{GSlemma}
Let $f $ be a smooth function on $\rr d$, $f\in C^{\infty}(\rr d).$ Then $f \in \maclS _{s}^{\sigma}(\rr d)$
(respectively $f \in \Sigma _{s}^{\sigma}(\rr d)$)
if and only if for every $\alpha \in \nn d$
\begin{equation} \label{GSspaces}
|\partial^\alpha f(x)| \lesssim l^{|\alpha|} \alpha!^\sigma  e^{-h|x|^{\frac 1s}}, \;\;\; x \in \rr d,
\end{equation}
for some $l, h>0$ (respectively for every $l,h>0$).
\end{lemma}

\medspace

The \emph{Gelfand-Shilov distribution spaces} $(\maclS _{s}^{\sigma})'(\rr d)$
and $(\Sigma _{s}^{\sigma})'(\rr d)$ are the projective and inductive limit
respectively of $(\maclS _{s;h}^{\sigma})'(\rr d)$:
\begin{equation}\tag*{(\ref{GSspacecond1})$'$}
(\maclS _{s}^{\sigma})'(\rr d) = \bigcap _{h>0}(\maclS _{s;h}^{\sigma})'(\rr d)\quad
\text{and}\quad (\Sigma _{s}^{\sigma})'(\rr d) =\bigcup _{h>0}(\maclS _{s;h}^{\sigma})'(\rr d).
\end{equation}
It follows that $\mascS '(\rr d)\hookrightarrow
(\maclS _s^\sigma)'(\rr d)$ when $s+\sigma \ge 1$, and if in addition
$(s,\sigma )\neq (\frac 12,\frac 12)$, then $(\maclS _s^\sigma)'(\rr d)
\hookrightarrow (\Sigma _s^\sigma)'(\rr d)$.

\par

The Fourier transform $\mathscr F$ is the linear and continuous map on $\mascS (\rr d)$,
given by the formula
$$
(\mathscr Ff)(\xi )= \widehat f(\xi ) \equiv (2\pi )^{-\frac d2}\int _{\rr
	{d}} f(x)e^{-i\scal  x\xi }\, dx, \;\;\; \xi \in \rr d,
$$
when $f\in \mascS (\rr d)$. Here $\scal \cdo \cdo$ denotes the usual
scalar product on $\rr d$.
The Fourier transform extends  uniquely to homeomorphisms
from $(\maclS _{s}^{\sigma})'(\rr d)$ to $(\maclS _{\sigma}^{s})'(\rr d)$,
and from  $(\Sigma _{s}^{\sigma})'(\rr d)$ to $(\Sigma _{\sigma}^{s})'(\rr d)$.
Furthermore, it restricts to homeomorphisms from
$\maclS _{s}^{\sigma}(\rr d)$ to $\maclS _{\sigma}^{s}(\rr d)$,
and from  $\Sigma _{s}^{\sigma}(\rr d)$ to $\Sigma _{\sigma}^{s}(\rr d)$.

\par

Next  we rewrite  the definition of Gelfand-Shilov spaces in the
notation which is convenient for our analysis, see also \cite{CaTo, GS, GZ}. We put
$$
\rr {d_0+\dots+d_k} = \rr {d_0} \times \rr {d_1} \times \dots \times \rr {d_k} = \rr {(d_0,\dots,d_k)} = \rr {\mathbf{d}}.
$$

\par

\begin{defn}
Let $k \in \mathbf N$, $ \mathbf{\sigma} = (\sigma_0,\dots,\sigma_k) >0$, $ \mathbf{s} = (s_0,\dots,s_k)>0$,
and $ \mathbf{d} = d_0 + \dots + d_k $.
The Gelfand-Shilov spaces
\begin{alignat*}{3}
&\maclS _\mathbf{s}  ^\mathbf{\sigma} (\rr {\mathbf{d}}) =  \maclS _{s_0,\dots,s_k}^{\sigma_0,\dots,\sigma_k}
(\rr {d_0+\dots+d_k}) &
&\quad \text{and} \quad &
&\Sigma _\mathbf{s}  ^\mathbf{\sigma} (\rr {\mathbf{d}}) =  \Sigma _{s_0,\dots,s_k}^{\sigma_0,\dots,\sigma_k}
(\rr {d_0+\dots+d_k}),
\end{alignat*}
consist of all
$F\in C^\infty (\rr {d_0+\dots+d_k})$ such that
$$
|x_0^{\alpha_0 }\dots x_k^{\alpha_k}\partial _{x_0}^{\beta _0}\dots
\partial _{x_k}^{\beta _k}F(x_0,\dots,x_k)|
\lesssim
h^{|\alpha _0+\beta _0+\dots+\alpha _k+\beta _k|}
\prod_{j=0}^{k}\alpha_j!^{s_j}\beta_j!^{\sigma_j}
$$
for some $h>0$ and  for every $h>0$ respectively, where $x_j\in \rr{d_j}$,
$\alpha_j,\beta_j\in \mathbf N^{d_j}$,
$j=0,\dots,k$.
The dual spaces of
$ \maclS _\mathbf{s}  ^\mathbf{\sigma} (\rr {\mathbf{d}}) $ and
$\Sigma _\mathbf{s}  ^\mathbf{\sigma} (\rr {\mathbf{d}})$
are denoted by
$$
(\maclS _\mathbf{s}  ^\mathbf{\sigma})' (\rr {\mathbf{d}})
=
(\maclS _{s_0,\dots,s_k}^{\sigma_0,\dots,\sigma_k})'
(\rr {d_0+\dots+d_k})
$$
and
$$
(\Sigma _{s_0,\dots,s_k}^{\sigma_0,\dots,\sigma_k})' (\rr {\mathbf{d}}) = (\Sigma _{s_0,\dots,s_k}^{\sigma_0,\dots,\sigma_k})'
(\rr{d_0+\dots+d_k})
$$
respectively.
\end{defn}

The space $ \maclS _\mathbf{s}  ^\mathbf{\sigma} (\rr {\mathbf{d}})$ is nontrivial if and only if $s_j + \sigma_j \geq 1$,
for each $ j = 0,\dots, k$ and
$ \Sigma _\mathbf{s}  ^\mathbf{\sigma} (\rr {\mathbf{d}})$ is nontrivial if and only if $s_j + \sigma_j \geq 1$,
and $(s_j,\sigma_j) \neq ( \frac{1}{2}, \frac{1}{2}) $ for each $ j = 0,\dots, k$.

\par

Obviously, if $\sigma_j=\sigma$, $s_j=s$ and $d_j=d$,
$j=0,\dots, k$, then
$$
\maclS_{s_0,\dots,s_k}^{\sigma_0,\dots,\sigma_k}
(\rr {d_0+\dots+d_k})
\equiv
\maclS _s^\sigma(\rr{(k+1)d})
,\quad
\Sigma_{s_0,\dots,s_k}^{\sigma_0,\dots,\sigma_k}
(\rr {d_0+\dots+d_k})
\equiv
\Sigma _s ^\sigma(\rr{(k+1)d}),
$$
$$
(\maclS _{s_0,\dots,s_k}^{\sigma_0,\dots,\sigma_k})'
(\rr {d_0+\dots+d_k})
\equiv
(\maclS _s^\sigma)'(\rr{(k+1)d}),
$$
and
$$
( \Sigma _{s_0,\dots,s_k}^{\sigma_0,\dots,\sigma_k})'
(\rr {d_0+\dots+d_k})
\equiv
(\Sigma _s ^\sigma)'(\rr{(k+1)d}).
$$

\par

The  Fourier transform is a homeomorphism between
$ \maclS _s ^\sigma(\rr{d})$ and $ \maclS ^s _\sigma(\rr{d})$
(and between $ \Sigma _s ^\sigma(\rr{d})$ and  $ \Sigma ^s _\sigma(\rr{d})$), cf. \cite{GS}. This, together with the kernel theorem for
Gelfand-Shilov spaces (see \cite{LCPT, Prang, Te3} implies
the following mapping properties
of partial Fourier transforms on
Gelfand-Shilov spaces. The proof is therefore omitted.
Here, $\mascF _jF$ is the partial Fourier transforms
of $F(x_0,x_1,\dots,x_{k})$
with respect to $x_j\in \rr {d_j}$, $j=0,\dots,k$.

\par

\begin{prop}\label{propBroadGSSpaceChar}
Let $k\in \mathbf N$, $s_j,\sigma _j>0$, $j=0,\dots,k$.
Then the following is true:
\begin{enumerate}
\item the mapping $\mascF _j$  on $\mascS (\rr {d_0+\dots+d_k})$
restrict to homeomorphism
\begin{align*}
\mascF _j \, &: \, \maclS  _{s_0,\dots,s_k}^{\sigma_0,\dots,\sigma_k}
(\rr {d_0+\dots+d_k})
\to
\maclS  _{s _0,\dots,s_{j-1},\sigma_j,s_{j+1},\dots,s_k}
^{\sigma _0,\dots,\sigma_{j-1},s_j,\sigma_{j+1},\dots,\sigma_k}
(\rr {d_0+\dots+d_k});
\end{align*}

\vrum

\item the mapping $\mascF _j$  on $\mascS (\rr {d_0+\dots+d_k})$
is uniquely extendable to homeomorphism
\begin{align*}
\mascF _j \, &: \, (\maclS  _{s_0,\dots,s_k}^{\sigma_0,\dots,\sigma_k})'
(\rr {d_0+\dots+d_k}) \to
(\maclS  _{s _0,\dots,s_{j-1},\sigma_j,s_{j+1},\dots,s_k}
^{\sigma _0,\dots,\sigma_{j-1},s_j,\sigma_{j+1},\dots,\sigma_k})'
(\rr {d_0+\dots+d_k}).
\end{align*}
\end{enumerate}
	
\par
	
The same holds true if the $\maclS _{s_0,\dots,s_k}^{\sigma_0,\dots,\sigma_k}$-spaces and their duals
are replaced by corresponding
$\Sigma _{s_0,\dots,s_k}^{\sigma_0,\dots,\sigma_k}$
-spaces and their duals in each occurrence.
\end{prop}

\par

The result analogous to Proposition \ref{propBroadGSSpaceChar} holds for
partial Fourier transforms with respect to some choice of variables. In particular the
(full) Fourier transform on $\mascS (\rr {\mathbf{d}})$ restricts to
to homeomorphism
\begin{align*}
\mascF  \, &: \, \maclS  _{\mathbf{s}}^{\mathbf{\sigma}}
(\rr {\mathbf{d}})
\to
\maclS  ^{\mathbf{s}} _{\mathbf{\sigma}}
(\rr {\mathbf{d}});
\end{align*}
and is uniquely extendable to homeomorphism
\begin{align*}
\mascF  \, &: \, (\maclS  _{\mathbf{s}}^{\mathbf{\sigma}}
)'
(\rr {\mathbf{d}}) \to
(\maclS  ^{\mathbf{s}} _{\mathbf{\sigma}}
)'
(\rr {\mathbf{d}}),
\end{align*}
and same holds true if $\maclS _{\mathbf{s}}^{\mathbf{\sigma}}
(\rr {\mathbf{d}})$ spaces and their duals
are replaced by corresponding $\Sigma _{\mathbf{s}}^{\mathbf{\sigma}}
(\rr {\mathbf{d}})$ spaces and their duals.

Alternatively, this result is contained in the following Proposition.

\begin{prop}\label{prop:GScharac}
Let $k\in \mathbf N$, $ \mathbf{\sigma} = (\sigma_0,\dots,\sigma_k) >0$, $ \mathbf{s} = (s_0,\dots,s_k)>0$,
and $ \mathbf{d} = d_1 + \dots + d_k $.
Then the following
conditions are equivalent.
\begin{enumerate}
\item $F\in \maclS  _{\mathbf{s}}^{\mathbf{\sigma}}
(\rr {\mathbf{d}}) $\quad
($F\in\Sigma _{\mathbf{s}}^{\mathbf{\sigma}}
(\rr {\mathbf{d}})$);

\vrum

\item for some $r>0$ (for every $r>0$) it holds
\begin{align*}
\displaystyle{|F(x_0,\dots,x_k)|\lesssim
	e^{-r\left(|x_0|^{\frac 1{s_0}}+\dots+|x_k|^{\frac 1{s_k}}
		\right)}}
\intertext{and}
\displaystyle{|\widehat F(\xi _0,\dots,\xi _k)|
	\lesssim
	e^{-r\left(|\xi _0|^{\frac 1{\sigma _0}}+\dots
		+ |\xi _k|^{\frac 1{\sigma _k}} \right)}}.
\end{align*}

\vrum

\item for every $\alpha = (\alpha_0, \dots, \alpha_k) \in \nn {\mathbf{d}}$ and for some $h,r>0$ (for every $h,r>0$) it holds
$$
|\partial ^{\alpha} F(x_0,\dots,x_k)|\lesssim
h ^{|\alpha|}	\prod_{j=0} ^k \alpha_j ^{\sigma_j}
e^{-r\left(|x_0|^{\frac 1{s_0}}+\dots+|x_k|^{\frac 1{s_k}}
		\right)}.
$$
\end{enumerate}
\end{prop}

\begin{proof}
The equivalence between (1) and (2)  follows from \cite{ChChKi96},
and (1) $ \Leftrightarrow $ (2) can be proved by a slight modification of the proof of Lemma \ref{GSlemma}
(cf. Appendix \ref{appendix}) and we therefore leave it for the reader.

\end{proof}

%
%

\subsection{Weight functions}\label{subsec1.1}

A function $\omega$ is called a \emph{weight}
or \emph{weight function}
on $\rr d$, if $\omega,1/\omega\in L^\infty _{loc}(\rr d)$
are positive everywhere. Without loss of generality we may assume that the weight functions are continuous on $\rr d$.
Let $\omega$ and $v$ be weights on $\rr d$.
Then $\omega$ is called \emph{$v$-moderate}
or \emph{moderate},
if
\begin{equation}\label{e1.1}
\omega (x_1+x_2)\lesssim \omega (x_1) v(x_2),\quad x_1,x_2\in \rr d .
\end{equation}

If $v$ can be chosen as polynomial, then $\omega$ is
called a weight of
polynomial type. A weight function $v$ is
\emph{submultiplicative}, if
it is symmetric in each coordinate and
$$
v (x_1+x_2)\lesssim v (x_1) v(x_2),\quad x_1,x_2\in \rr d .
$$
From now on, $v$ always denote a submultiplicative
weight if nothing else is stated.
In particular, if \eqref{e1.1} holds and $v$ is submultiplicative,
then
\begin{equation}\label{eq:2Next}
\frac {\omega (x_1)}{v(x_2)}
\lesssim \omega(x_1 + x_2) \lesssim \omega(x_1)v(x_2),
\quad x_1,x_2\in\rr{d}.
\end{equation}

\par

If $\omega$ is a moderate weight on $\rr d$, then there exists a submultiplicative weight
$v$ on $\rr d$ such that \eqref{e1.1} and \eqref{eq:2Next}
hold, cf. \cite{To8,To11,To18}.
Moreover if $v$ is submultiplicative on $\rr d$, then
\begin{equation}\label{Eq:CondSubWeights}
1\lesssim v(x) \lesssim e^{c|x|}
\end{equation}
for some constant $c>0$ (cf. \cite[Lemma 4.2]{Gc2.5}).

In particular, if $\omega$ is moderate,
then
$$
\omega (x+y)\lesssim \omega (x)e^{c|y|}
\quad \text{and}\quad
e^{-c|x|}\lesssim \omega (x)\lesssim e^{c|x|},\quad
x,y\in \rr d
$$
for some $c>0$.

\par

For a given $k\in \mathbf N$, we let $\mascP _E(\rr {d_0+\dots+d_k})$
be the set of all moderate weights on
$\rr {d_0+\dots+d_k}$, and $\mascP (\rr {d_0+\dots+d_k})$
be the subset of $\mascP  _E(\rr {d_0+\dots+d_k})$
which consists of weights of  polynomial type.

If $\omega \in \mascP _E(\rr {d_0+\dots+d_k})$
then there exists a submultiplicative weight $v$ on $\rr {d_0+\dots+d_k}$,
such that
\begin{equation}\label{eq:2Nextbis}
\frac {\omega (x_0, \dots,x_k)}{v(y_0, \dots,y_k)}
\lesssim \omega(x_0+y_0,  \dots,x_k+y_k)
\lesssim \omega(x_0, \dots, x_k)v(y_0, \dots,y_k),
\end{equation}
where  $x_j, y_j \in\rr{d_j}$, $j=0,\dots,k$.
Note that \eqref{Eq:CondSubWeights} for a submultiplicative weight $v$
on $\rr {d_0+\dots+d_k}$ becomes
\begin{equation}\label{eq:expEstonsubmultipicative}
1\lesssim v(x_0,\dots,x _k) \lesssim e^{r(|x_0|+\dots+|x _k|)},
\quad x_j\in \rr {d_j},\ j=0,\dots,k,
\end{equation}
for some $r>0$.

\par

Next we extend the weight functions considered in \cite{AbCaTo,AbCoTo} to the case  when
$\rr {d} = \rr {d_0+\dots+d_k}$.

\begin{defn} \label{classesofweights}
Let $k\in \mathbf N$ and  $s_j>0$, $j=0,\dots,k$. Then, the set $\mascP _{s_0,\dots,s_k}
(\rr {d_0+\dots+d_k})$
($\mascP _{s_0,\dots,s_k}^0(\rr{d_0+\dots+d_k})$) consists of
all weights $\omega\in \mascP _E(\rr {d_0+\dots+d_k})$ such that
\begin{equation}\label{eq:omegaEst}
\omega(x_0+y_0, \dots,x_k+y_k)
\lesssim \omega (x_0, \dots,x_k )
e^{r(|y_0|^{\frac 1{s_0}}+\dots+|y_k |^{\frac 1{s_k}})},\;
x_j,y_j\in \rr{d_j}
\end{equation}
holds for some (for every) $r>0$.
\end{defn}

In particular, if $\omega \in  \mascP _{s_0,\dots,s_k}
(\rr {d_0+\dots+d_k})$
($\mascP _{s_0,\dots,s_k}^0(\rr{d_0+\dots+d_k})$), then
$$
e^{-r(|x_0|^{\frac 1{s_0}}+\dots+|x_k |^{\frac 1{s_k}})}
\lesssim
\omega (x_0, \dots,x_k )
\lesssim
e^{r(|x_0|^{\frac 1{s_0}} +\dots+|x_k |^{\frac 1{s_k}})}
$$
for some $r>0$ (for every $r>0$).

By \eqref{eq:2Nextbis} and \eqref{eq:expEstonsubmultipicative} it follows that
$$
\mascP _{s_0,\dots,s_k}^0(\rr{d_0+\dots+d_k}) =
\mascP _{\tilde s_0,\dots, \tilde s_k} (\rr{d_0+\dots+d_k}) =
\mascP _E (\rr {d_0+\dots+d_k})
$$
when $s_j<1$ and $\tilde s_j\le 1$, $j=0,\dots,k$. For convenience we set
$$\mascP^0_E(\rr{d_0+\dots+d_k})=\mascP^0_{E,1}(\rr{d_0+\dots+d_k}).$$

\par

The following  extension of \cite[Proposition 1.6]{AbCoTo},
shows that for any weight in $\mascP _E (\rr {d_0+\dots+d_k})$,
there are equivalent weights that satisfy the anisotropic Gevrey regularity.

\par

\begin{prop}\label{Prop:EquivWeights}
Let there be given $\omega \in \mascP _E(\rr {d_0+d_1+\dots+d_k})$.
Then there exists a weight
$\omega _0\in \mascP _E(\rr{d_0+d_1+\dots+d_k})
\cap C^\infty (\rr {d_0+d_1+\dots+d_k})$
such that the following is true:
\begin{enumerate}	
\item $\omega _0\asymp \omega $;

\vrum

\item for every (multiindex) $\alpha_j\geq 0$, $ j =0,\dots,k,$ we have
\begin{multline*}
|\partial _{x} ^{\alpha_0}\partial _{\xi_1} ^{\alpha_1}
\dots \partial _{\xi_k} ^{\alpha_k} \omega_0 (x, \xi_1,\dots,\xi_k)|
\lesssim
h^{|\alpha_0 +\alpha_1+\dots+\alpha_k|}\prod_{j=0} ^k \alpha_j !^{s_j}
\omega (x, \xi_1,\dots,\xi_k)
\\[1ex]
\asymp
h^{|\alpha_0 +\alpha_1+\dots+\alpha_k|}\prod_{j=0} ^k \alpha_j !^{s_j} \omega_0 (x, \xi_1,\dots,\xi_k),
\;\;\;
x\in \rr d_0, \; \xi_j \in \rr{d_j},\; j=1,\dots, k,
\end{multline*}
for every $h>0$ and $ s_j>0$, $ j =0,\dots,k$.
\end{enumerate}		
\end{prop}

\par

The proof is given in  Appendix  \ref{appendix}.

%
%
%

\subsection{Modulation spaces}\label{subsec1.3}
Modulation spaces, originally introduced by Feichtinger in \cite{Fe4},
are recognized as appropriate family of spaces when dealing with problems of time-frequency analysis, see
\cite{Fe4,FG1,FG2,FG4,GaSa,Gro,RSTT,Te2}, to mention just a few references. A broader family of
modulation spaces is recently studied in \cite{AbCoTo, PfeuToft, To25}.

\par

Let $s,\sigma>0$, such that $s+\sigma\geq 1$, and let
$\phi \in \maclS _s ^\sigma (\rr d)$ be fixed. Then the \emph{short-time Fourier transform} $V_\phi f$ of $f\in (\maclS _s^\sigma )'
(\rr d)$ with respect to the  \emph{window function} $\phi$ is defined by
$$
V_\phi f(x,\xi ) \equiv  
\mascF (f \, \cdot \, \overline {\phi (\cdo -x)})(\xi),\;\;\; x,\xi \in \rr d.
$$
This definition makes sense as a Gelfand-Shilov distribution \cite[Remark 1.5]{AbCaTo}.

\par

If $f ,\phi \in \maclS _s^\sigma (\rr d)$, then
$$
V_\phi f(x,\xi ) = (2\pi )^{-\frac d2}\int f(y)\overline {\phi
	(y-x)}e^{-i\scal y\xi}\, dy .
$$

\par

Let $s,\sigma>0$, such that $s+\sigma\geq 1$.
Let there be given $\phi \in \maclS _s^\sigma(\rr d)\setminus 0$, $p,q\in [1,\infty ]$
and $\omega \in\mascP _E(\rr {2d})$. Then the
\emph{modulation space} $M^{p,q}_{(\omega )}(\rr d)$ consists of all
Gelfand-Shilov distributions $f$ on $\rr d$ such that
\begin{equation}\label{modnorm}
\nm f{M^{p,q}_{(\omega )}} \equiv \Big ( \int \Big ( \int |V_\phi f(x,\xi
)\omega (x,\xi )|^p\, dx\Big )^{q/p} \, d\xi \Big )^{1/q} <\infty
\end{equation}
(with the obvious changes if $p=\infty$ and/or
$q=\infty$). If $p=q$ we simply write $M^p_{(\omega )}$
instead of $M^{p,p}_{(\omega )}$, and
if $\omega =1$, then we set $M^{p,q}=M^{p,q}_{(\omega )}$
and $M^{p}=M^{p}_{(\omega )}$.

\par

The spaces $M_{(\omega)}^{p,q}$ are Banach spaces
and every  $\phi \in  M^r_{(v)} \setminus 0$
yields an equivalent norm in \eqref{modnorm} and so
$M_{(\omega)}^{p,q}$ is independent on the choice of
$\phi \in  M^r_{(v)}$ \cite[Proposition 1.1]{To25}.

\par

Gelfand-Shilov spaces and their dual spaces can be described as projective or inductive limits of modulation spaces  \cite[Theorem 3.9]{To18}. In particular, we have the following characterization of Gelfand-Shilov spaces by the means of
the short-time Fourier transform. We refer to \cite{GZ} for the proof, see also
\cite{Te1, Te6, To18}.

\par

\begin{prop}\label{Prop:STFTGelfand2}
Let $k \in \mathbf N$, $ \mathbf{\sigma} = (\sigma_0,\dots,\sigma_k) >0$, $ \mathbf{s} = (s_0,\dots,s_k)>0$,
and $ \mathbf{d} = d_1 + \dots + d_k $.
Also let $\phi \in \mathcal S_{\mathbf{s}}^{\mathbf{\sigma} }(\rr {\mathbf{d}})
\setminus 0$. Then the following is true:
\begin{enumerate}
\item $F \in  \maclS _{\mathbf{s}}^{\mathbf{\sigma} }(\rr {\mathbf{d}})$
if and only if
\begin{equation}\label{stftexpest2}
|V_\phi F(x_0,\dots,x_k,\xi _0,\dots,\xi _k )|
\lesssim
e^{-r \left(|x_0|^{\frac 1{s_0}}+\dots+|x_k|^{\frac 1{s_k}} + |\xi _0|^{\frac 1{\sigma _0}}+\dots
		+ |\xi _k|^{\frac 1{\sigma _k}} 		\right)}
\end{equation}
holds for some $r > 0$;

\vrum

\item if, in addition,
$\phi \in \Sigma _{\mathbf{s}}^{\mathbf{\sigma} }(\rr {\mathbf{d}}) \setminus 0$, then
$F\in  \Sigma _{\mathbf{s}}^{\mathbf{\sigma} }(\rr {\mathbf{d}}) $
if and only if \eqref{stftexpest2} holds for every $r > 0$.
\end{enumerate}
\end{prop}

\subsection{Symbol classes and Pseudo-differential operators}\label{sec:PrePdo}

First we introduce function spaces related to symbol classes
of the multilinear pseudo-differential operators.
We consider  $a \in C^{\infty} (\rr {d_0+\dots+d_k})$ which obey various conditions of the form
\begin{multline*}
|\partial _x^{\alpha} \partial _{\xi_1}^{\beta_1}
\dots, \partial _{\xi_k}^{\beta_k}
a(x,\xi_1,\dots,\xi_k )|
\\[1ex]
\lesssim
h ^{|\alpha +\beta_1+\dots+\beta_k|}
\alpha !^\sigma
\prod_{j=1} ^k \beta_j !^{s_j}
\cdot\omega (x,\xi_1,\dots,\xi_k),
\end{multline*}
 $w \in \mascP _E (\rr{d_0+d_1+\dots+d_k})$,
$\alpha\in\nn{d_0}$, $\beta_j\in \rr{d_j}$, $s_j,\sigma,h>0$, $j = 1,\dots, k$.

\par

When $k=1$ we recover the condition (1.14) from \cite{AbTo}.
Similarly to \cite{AbTo}, for a given $w \in \mascP _E (\rr{d_0+d_1+\dots+d_k})$
and $s_j,\sigma,h>0$, $j = 1,\dots, k$, we consider norms of the form
\begin{multline} \label{norm}
\nm a{\Gamma _{(\omega)}^{\sigma,\mathbf{s};h}}
\equiv
\\
\sup _{\substack{\alpha \in \nn{d_0}\\\beta_j \in \nn {d_j}}}
\left ( \sup _{\substack{x\in \rr {d_0}\\\xi_j\in\rr{d_j}}}
\left (
\frac {|\partial _x^{\alpha}
	\partial _{\xi_1}^{\beta_1}\dots, \partial _{\xi_k}^{\beta_k}
	a(x,\xi_1,\dots,\xi_k )|}
   {h ^{|\alpha +\beta_1 +\dots+\beta_k|}
	\alpha  !^\sigma
	\prod_{j=1} ^k \beta_j !^{s_j} \cdot \omega (x,\xi_1,\dots,\xi_k)}
\right )
\right ).
\end{multline}

\par

More precisely, we are interested in invariance and
continuity for bilinear pseudo-differential operators when symbols belong to the following symbol classes.

\par

\begin{defn}\label{Def:GammaSymb2}
Let there be given  $  \sigma, s_j, h>0$, $j=1,\dots,k$ and $\omega \in \mascP _E (\rr{d_0+d_1+\dots+d_k})$, and set $\mathbf{s} = (s_1, \dots, s_k)$.

\begin{enumerate}	
\item The set ${\Gamma _{(\omega)}^{\sigma, \mathbf{s};h}}
(\rr {d_0+\dots+d_k})$
consists of all $a \in C^\infty(\rr {d_0+\dots+d_k})$ such that
$$\nm a{\Gamma _{(\omega)}^{\sigma,\mathbf{s};h}} < \infty,$$
where the norm $\nm \cdot{\Gamma _{(\omega)}^{\sigma,\mathbf{s};h}} $ is given by \eqref{norm}.

\vrum

\item The sets ${\Gamma _{(\omega)}^{\sigma,\mathbf{s}}}
(\rr {d_0+\dots+d_k})$
and ${\Gamma _{(\omega)}^{\sigma,\mathbf{s};0}}
(\rr {d_0+\dots+d_k})$
are given by
\begin{align*}
{\Gamma _{(\omega)}^{\sigma,\mathbf{s} }}
(\rr {d_0+\dots+d_k})
\equiv
\bigcup _{h>0}
{\Gamma _{(\omega)}^{\sigma,\mathbf{s};h}}
(\rr {d_0+\dots+d_k})
\intertext{and}
{\Gamma _{(\omega)}^{\sigma,\mathbf{s} ;0}}
(\rr {d_0+\dots+d_k})
\equiv
\bigcap _{h>0}
{\Gamma _{(\omega)}^{\sigma,\mathbf{s};h}}
(\rr {d_0+\dots+d_k}),
\end{align*}
and their topologies are, respectively,
the inductive and the projective limit topologies
of ${\Gamma _{(\omega)}^{\sigma,\mathbf{s};h}} (\rr {d_0+\dots+d_k})$ with respect to $h>0$.				
\end{enumerate}
\end{defn}

\par

As it is common in the theory of ultradifferentiable functions, we say that (the inductive limit)
${\Gamma _{(\omega)}^{\sigma,\mathbf{s}}} (\rr {d_0+\dots+d_k})$ is
a Roumieu class,  and (the projective limit) ${\Gamma _{(\omega)}^{\sigma,\mathbf{s};0}}
(\rr {d_0+\dots+d_k})$ is a Beurling  class of test functions.

Notice that ${\Gamma _{(\omega)}^{\sigma,\mathbf{s}}}
(\rr {d_0+\dots+d_k})$ and ${\Gamma _{(\omega)}^{\sigma,\mathbf{s};0}}
(\rr {d_0+\dots+d_k})$ are nontrivial for any $ \sigma, s_j, h>0$, $j=1,\dots,k$.
For instance by Proposition \ref{Prop:EquivWeights} for any
$\omega \in \mascP _E(\rr{d_0+d_1+\dots+d_k})$ there exist a smooth function
$\omega_0 \in \mascP _E(\rr{d_0+d_1+\dots+d_k}) $ such that
$\omega_0 \in {\Gamma _{(\omega)}^{\sigma,\mathbf{s};0}} (\rr {d_0+\dots+d_k})$.

\par

When $k=1$ we put $\sigma_1 =\sigma $ and recover the symbol classes
$ {\Gamma _{(\omega)}^{\sigma, s }} (\rr {d_0+d_1}) $ and
$ {\Gamma _{(\omega)}^{\sigma, s;0 }} (\rr {d_0+d_1}) $ considered in \cite{AbTo}.

\par

Next we recall some facts on pseudo-differential operators.
The pseudo-differential
operator $\op _t (a)$ is the linear and continuous operator on $\maclS (\rr d)$,
defined by the formula
$$
\op _t (a)f(x) 
= \frac{1}{(2\pi  ) ^{d}}\iint a(x- t(x-y),\xi )f(y)e^{i\scal {x-y}\xi }\,
dyd\xi, \;\;\; x \in \rr d .
$$
More generally, the definition of $\op _t (a)$ extends uniquely to
$a\in \maclS '(\rr {2d})$,
and then $\op _t (a)$ is continuous from
$\maclS (\rr d)$ to $\maclS '(\rr d)$.

\par

Let $\mathbf{t} = (t_1, t_2, \dots, t_m) \in [0,1]^m,$  be such that $ \sum_{j=1} ^m t_j \leq 1,$
and put $ \vec{f} = (f_1, f_2, \dots, f_m) \in \mascS (\rr {md})$.
The multilinear pseudo-differential operator
$\op_{\mathbf{t}}(a)$  from $\mascS (\rr {md})$ to $\mascS' (\rr d)$
is defined by the formula
$$
\op_{\mathbf{t}}(a) \vec{f} (x)=
\frac{1}{(2\pi)^{2md}}\iint e^{-i\psi(x,\mathbf{y},\mathbf{\xi})}
a_{\mathbf{t}}(x,\mathbf{y},\mathbf{\xi})\prod_{j=1} ^m f_j(y_j)
d\mathbf{y} d\mathbf{\xi},
$$
where
$$
a_{\mathbf{t}}(x,\mathbf{y},\mathbf{\xi}) =a(x+\sum_{j=1} ^m ( t_j y_j -x), \xi, \eta), \;\;\;
x, y_j, \xi_j \in \rr d,
$$
and the phase function $\psi$ is defined by
$$
\psi(x,\mathbf{y},\mathbf{\xi})= \sum_{j=1} ^m \scal{y_j -x}{\xi_j},
\;\;\; x,y_j,\xi_j \in \rr d.
$$

\par

When $m=2$ we obtain  bilinear pseudo-differential operators
$\op_{r, t}(a)$. That is, $\op_{r, t}(a)$ is the bilinear
and continuous operator from $\mascS (\rr d)\otimes\mascS (\rr d)$ to $\mascS' (\rr d)$,
defined by the formula
\begin{multline}\label{eq:BilPdo}
\left(\op_{r, t}(a)(f,g)\right)(x)=
\\[1ex]
(2\pi)^{-2d}\iiiint e^{-i\psi(x,y,z,\xi,\eta)}
a_{r,t}(x,y,z, \xi, \eta) f(y)g(z) \,dydzd\xi d\eta, \; x \in \rr d,
\end{multline}
where $(r,t)\in[0,1]\times[0,1]$, $ r+t\leq 1$,
$$
a_{r,t}(x,y,z, \xi, \eta)=a(x+r(y-x)+t(z-x), \xi, \eta), \;\;\;
x,y,z, \xi, \eta \in \rr d,
$$
and the phase function $\psi$ is defined by
$$
\psi(x,y,z,\xi,\eta)=\scal{y-x}\xi+\scal{z-x}\eta,
\;\;\; x,y,z, \xi, \eta \in \rr d.
$$

\par

If $r=t=0$, then the definition of  $\op_{0}(a)$ coincides
with the definition of bilinear pseudo-differential operators
$$
T_a(f,g)(x)=(2\pi)^{-d}\iint e^{i\scal x{\xi+\eta}}a(x,\xi,\eta)
\widehat{f}(\xi)\widehat{g}(\eta)\, d\xi d\eta, \;\;\; x \in \rr d,
$$
considered in e.g \cite{BeGrHeOk}, and the corresponding
multilinear extension is studied in \cite{MOPf}.

\par

In fact, in Sections \ref{sec2} and  \ref{sec3} we will consider the action of $\op_{r, t}(a)$ when restricted to different Gelfand-Shilov spaces,
and related unique extension of such operators to Gelfand-Shilov distributions.

\par

We will use the following results about the continuity of linear pseudo-differential operators
with symbols in Gevrey-H\"ormander classes, and we refer to \cite[Theorem 2.1]{AbTo} and \cite[Theorem 3.7]{AbCaTo} respectively, for the proofs.

\par

\begin{prop}\label{thm:anisCntp}
Let
$s,\sigma \ge 1$, $p,q\in [1,\infty]$,
$\omega ,\omega _0\in\mascP _{s,\sigma}^0(\rr {2d})$,
and $a\in \Gamma _{(\omega_0 )}^{\sigma,s}(\rr {2d})$.
Then $\op _t (a)$ is a continuous operators from
$M^{p,q}_{(\omega _0\omega)} (\rr d)$ to $M^{p,q}_{(\omega)} (\rr d)$ for any $t \in [0,1]$.		
\end{prop}

Note that, in the notation of Definition \ref{classesofweights}
we have $\mascP _{s,\sigma}^0(\rr {2d}) = \mascP _{s_0,s_1}^0(\rr {d_0 + d_1})$ when $s_0 =s, $ $s_1 = \sigma$ and
$d_1 = d_2 =d.$

\par

\begin{prop}\label{Thm:theorem2}
Let $A \in \mathbf M (d,R)$, $s,\sigma >0$ be such that $s+\sigma \ge 1$,
$\omega\in\mascP _{s,\sigma}^0(\rr {2d})$ and let
$a\in \Gamma _{0}^{\sigma ,s;h}(\rr {2d})$ for some $h>0$. Then $\op_A (a)$
is continuous on $\maclS _s^\sigma (\rr d)$ and on
$(\maclS _s^\sigma )'(\rr d)$.
\end{prop}

\section{Characterization and invariance property for
	bilinear pseudo-differential operators}\label{sec2}

Our aim in this section is to show that $\Gamma ^{ \sigma,s_1,s_2}_{(\omega )} (\rr {3d})$
and $\Gamma ^{ \sigma, s_1,s_2;0}_{(\omega )} (\rr {3d})$ can be characterized
in terms of estimates of short-time Fourier transforms
and modulation spaces. This is done in subsection \ref{subsec2.2}. We refer to \cite{AbCaTo,To24} for similar results related to ``standard''  pseudo-differential operators. As a preparation, we show  that $\op_{r, t}(a)$ is independent of the choice of $r$ and $t$,
which gives the invariance property for bilinear operators, Theorem \ref{Thm:CalculiTransf}.
The counterpart of  Theorem \ref{Thm:CalculiTransf} for ``standard'' pseudo-differential is proved in e.g. \cite{AbCaTo, AbTo, To24}.
The key tools we employ to achieve the desired characterizations and invariance properties
in this section are mapping results for exponentials of certain linear operators,
similar to the typical ones often appearing in the ``usual''
Weyl-H\"ormander calculus, whose description is given here below.

\subsection{Mapping properties of exponential-type operators on Gel\-fand-Shilov spaces} For the study of mapping properties of the operator
$e^{-i\scal {rD_\xi+tD_\eta}{D_x}}$
we need the following auxiliary result.
By $\mathbf M(d,\re)$ we denote the set of all $d\times d$-matrices with entries in $\re$.

\begin{lemma}\label{lem:FTInvOpr}
Let $A,B\in \mathbf M(d,\re)$
and $a \in \mascS(\rr {3d})$. Then
\begin{equation}\label{eq:linid}
\big (\mascF _{2,3}^{-1}(e^{i \scal {AD_\xi+BD_\eta}{D_x}} a)\big )(x+Ay+Bz,y,z) =
(\mascF _{2,3}^{-1}a)(x,y,z),
\end{equation}
$ x,y,z \in \rr d.$
\end{lemma}

\par

\begin{proof}
Throughout the proof, the integrals are observed as
either Fourier transforms or inverse Fourier transforms of appropriate distributions.
The left-hand side of \eqref{eq:linid} is given by
\begin{multline*}
\big (\mascF _{2,3}^{-1}(e^{i \scal {AD_\xi+BD_\eta}{D_x}} a)\big )(x+Ay+Bz,y,z)
\\[1ex]
=
\iint e^{i(\scal y\xi+\scal z\eta)}\left(e^{i \scal {AD_\xi+BD_\eta}{D_x}} a)\right )
(x+Ay+Bz,\xi,\eta) \,d\xi d\eta,
\end{multline*}
and
\begin{multline*}
\left(e^{i \scal {AD_\xi+BD_\eta}{D_x}} a)\right )
(x+Ay+Bz,\xi,\eta)
\\[1ex]
=
\iiint e^{i(\scal{x+Ay+Bz} {\zeta}+\scal {y_1}{\xi}+\scal {z_1}\eta)}
e^{i (\scal{Ay_1+Bz_1}{\zeta})} \widehat{a}(\zeta,y_1,z_1)\, d\zeta dy_1dz_1,
\end{multline*}
where
\begin{equation*}
 \widehat{a}(\zeta,y_1,z_1)
\\[1ex]
=
\iiint e^{-i(\scal{x_1}\zeta+\scal{y_1}{\xi _1}+\scal {z_1}{\eta _1})}
a(x_1,\xi_1,\eta_1) \, dx_1 d\xi _1 d\eta _1,
\end{equation*}
$ x_1,\xi_1,\eta_1 \in \rr d$.
Let $\Psi\equiv \Psi(x,x_1,y,y_1,z,z_1,\zeta,\xi,\eta,\xi _1, \eta _1)$
be given by
\begin{multline*}
\Psi=\scal {y+y_1}{\xi}+\scal {z+z_1}{\eta}
+ \scal{x+Ay+Bz} {\zeta}
\\
+{ (\scal{Ay_1+Bz_1}{\zeta})}
{-(\scal{x_1}\zeta+\scal{y_1}{\xi _1}+\scal {z_1}{\eta _1})},
\end{multline*}
$x,x_1,y,y_1,z,z_1,\zeta,\xi,\eta,\xi _1, \eta _1 \in \rr d$.
It follows that
\begin{multline} \label{eq:FTInvOpr2}
\big (\mascF _{2,3}^{-1}(e^{i \scal {AD_\xi+BD_\eta}{D_x}} a)\big )(x+Ay+Bz,y,z)
\\[1ex]
=
\iiiint\!\!\!\iiiint
 e^{i\Psi(x,x_1,y,y_1,z,z_1,\zeta,\xi,\eta,\xi _1, \eta _1)}
 \\
\times a(x_1,\xi_1,\eta_1) \, dx_1 d\xi _1 d\eta _1
\, d\zeta dy_1dz_1
 \,d\xi d\eta, \;\;\; x,y,z \in \rr d.
\end{multline}
Since
$$
\int e^{i\scal {y+y_1}{\xi}} \, d\xi=\delta({y+y_1}),
\quad \text { and }
\int e^{i\scal {z+z_1}{\eta}} \, d\eta=\delta ({z+z_1}),
$$
where $\delta$ is the Dirac delta distribution,
it follows that \eqref{eq:FTInvOpr2}  reduces to
\begin{align*}
\iiiint
&e^{i(
\scal{x} {\zeta}
{-\scal{x_1}\zeta+\scal{y}{\xi _1}+\scal {z}{\eta _1}})
}
a(x_1,\xi_1,\eta_1) \, d\zeta dx_1 d\xi _1 d\eta _1
\\[1ex]
&=\iint
e^{i(\scal{y}{\xi _1}+\scal {z}{\eta _1})}
a(x,\xi_1,\eta_1) \,  d\xi _1 d\eta _1
=\left(\mascF^{-1}_{2,3}a\right)(x,y,z),
\end{align*}
and the claim follows.
\end{proof}

\par

Next we show some mapping properties
of the operator $e^{-i\scal {rD_\xi+tD_\eta}{D_x}}$ which are
an important ingredient in our analysis.

\par

\begin{thm}\label{Thm:CalculiTransf}
Let  $s_j ,\sigma _j$,  $j=1,2,3$, be such that
\begin{equation} \label{eq:conditions}
s_j+\sigma _j\ge 1,\;
\quad 0< s_2,\, s _3\le s_1,
\quad \text{and} \quad 0< \sigma _1\le \sigma _2,\, \sigma _3
\end{equation}
and let $ r,t\in[0,1]$ be such that $r+t\leq 1$.
Then the following is true:	
\begin{enumerate}		
\item $e^{-i\scal {rD_\xi+tD_\eta}{D_x}}$ on
$\mascS (\rr {3d})$ restricts to
a homeomorphism on $\maclS _{s_1,\sigma _2,\sigma _3}^{\sigma _1,s_2,s_3}
(\rr {3d})$,
and extends uniquely to a homeomorphism on
$(\maclS _{s_1,\sigma _2,\sigma _3}^{\sigma _1,s_2,s_3})'(\rr {3d})$;

\vrum

\item if in addition $(s_j,\sigma _j)\neq (\frac 12 ,\frac 12)$,
$j=1,2,3$, then $e^{-i\scal {rD_\xi+tD_\eta}{D_x}}$
on $\mascS (\rr {3d})$ restricts to a homeomorphism on
$\Sigma _{s_1,\sigma _2,\sigma _3}^{\sigma _1,s_2,s_3}(\rr {3d})$,
and extends uniquely to a homeomorphism on
$(\Sigma _{s_1,\sigma _2,\sigma _3}^{\sigma _1,s_2,s_3})'(\rr {3d})$.

\end{enumerate}		
\end{thm}

\par

\begin{proof}
We only prove (1) and leave (2) for the reader.

\par

Let $a\in \mascS (\rr {3d})$ and let  $U_{r,t}$ be the map given by
$$
(U_{r,t} F)(x,y)=F(x-ry-tz,y,z), \;\;\; x,y \in \rr d.
$$
By Lemma \ref{lem:FTInvOpr}, we have
$$
\big (\mascF _{2,3}^{-1}(e^{i \scal {rD_\xi+tD_\eta}{D_x}} a)\big )(x+ry+tz,y,z) =
(\mascF _{2,3}^{-1}a)(x,y,z), \;\;\; x,y,z \in \rr d,
$$
wherefrom $e^{i \scal {rD_\xi+tD_\eta}{D_x}} = \mascF _{2,3}\circ U_{r,t} \circ \mascF _{2,3}^{-1}$.
Therefore it only remains to show that the mapping $U_{r,t}$  is continuous on
$\maclS _{s_1,s_2,s_3} ^{\sigma_1,\sigma _2,\sigma _3}$.

Since the Fourier transform with respect to the 2$^{nd}$
and 3$^{rd}$ variables switches between the corresponding decay and regularity properties
on Gelfand-Shilov spaces we consider
$G=U_{r,t}F$, where $F\in \maclS _{s_1,\sigma_2,\sigma_3} ^{\sigma_1,s_2,s_3}$.
Then
$$
G(x,y,z)= F(x-ry-tz,y,z)
\quad \text{and}\quad
\widehat G(\zeta,\xi ,\eta )= \widehat F(\zeta ,\zeta+r\xi,\zeta +t\eta ),
$$
$x,y,z, \zeta,\xi ,\eta \in \rr d$.
In view of  Proposition \ref{propBroadGSSpaceChar} and from the assumptions on
$s_j$ and $\sigma _j$, it follows that there exist constants $c,r_0>0$,
where $c$ depends on $r,t$, $s_j$ and $\sigma _j$ only, such that
\begin{multline*}
|G(x,y,z)|= |F(x-ry-tz,y,z)|
\\
\lesssim e^{-r_0(|x-ry-tz|^{\frac 1{s_1}}
	+|y|^{\frac 1{s_2} }+|z|^{\frac 1{s_3} } ) }
\lesssim
e^{-cr_0(|x|^{\frac 1{s_1}}+|y|^{\frac 1{s_2}}
+|z|^{\frac 1{s_3}})},
\end{multline*}
$x,y,z \in \rr d,$ and
\begin{multline*}
|\widehat G(\zeta,\xi ,\eta )|= |\widehat F(\zeta ,\zeta+r\xi,\zeta +t\eta )|
\\
\lesssim e^{-r(|\zeta |^{\frac 1{\sigma _1}}+|\zeta+r\xi | ^{\frac 1{\sigma_2}}
|\zeta+t\eta| ^{\frac 1{\sigma_3}})}
\lesssim
e^{-cr(|\zeta |^{\frac 1{\sigma_1}}+|\xi |^{\frac 1{\sigma_2}}
+|\eta|^{\frac 1{\sigma _3}})},
\end{multline*}
$ \zeta,\xi ,\eta \in \rr d$. The result follows since by Proposition \ref{prop:GScharac}  the topology in
$\maclS _{s_1,s_2,s_3} ^{\sigma_1,\sigma _2,\sigma _3} (\rr {3d})$ can be
defined by the above estimates.

\end{proof}

\par

\begin{cor}
Let $s,\sigma >0$ be such that $s+\sigma \ge 1$ and
$\sigma \le s$. Then $e^{-i\scal {rD_\xi+tD_\eta}{D_x}}$ is
a homeomorphism on $\maclS _{s}^{\sigma}(\rr {3d})$,
$\Sigma _{s}^{\sigma}(\rr {3d})$,
$(\maclS _{s}^{\sigma})'(\rr {3d})$ and on
$(\Sigma _{s}^{\sigma})'(\rr {3d})$.	
\end{cor}

\par

Next we study an invariance property of bilinear
pseudo-differential operators $\op_{r, t}(a)$ given by \eqref{eq:BilPdo}.
More precisely, it can be shown that for every Gelfand-Shilov distribution $a$
there is a unique distribution $b$ in the same Gelfand-Shilov class
such that $\op _{r_1,t_1}(a) = \op _{r_2,t_2} (b)$, when $r_j,t_j\in[0,1]$ and $r_j+ t_j\leq 1$.
The following result, which explains the
relation between such $a$ and $b$, follows from
Theorem \ref{Thm:CalculiTransf} when the conditions in \eqref{eq:conditions} are fulfilled.
We give an independent proof in Appendix  \ref{appendix}.

\par

\begin{prop}\label{Prop:CalculiTransfer}
Let  $r_j,t_j\in[0,1]$ be such that $r_j+ t_j\leq 1$, and
let $a,b \in (\maclS _{s_1,\sigma_2,\sigma_3} ^{\sigma_1, s_2,s_3} )'(\rr {3d})$, where
$s_j ,\sigma _j > 0$, and $s_j +\sigma _j \geq 1$,  $j=1,2,3.$
Then
\begin{align}\label{calculitransform}
\op _{r_1,t_1}(a) &= \op _{r_2,t_2}(b)\nonumber
\\
\quad &\Leftrightarrow \quad
\\
e^{-i\scal{r_1D_\xi+t_1D_\eta}{D_x}}a(x,\xi,\eta )
&=	e^{-i\scal{r_2D_\xi+t_2D_\eta}{D_x}} b(x,\xi,\eta ), \;\;\;
x,\xi,\eta \in \rr {d}.\nonumber
\end{align}
\end{prop}

\par

Note that the latter equality in \eqref{calculitransform} makes sense since it is equivalent to
$$
e^{-i\scal{r_1y+t_1z}{\zeta}}\widehat a(\zeta ,y,z)
=e^{-i\scal{r_2y+t_2z}{\zeta}}\widehat b(\zeta ,y,z), \;\;\; \zeta ,y,z \in \rr {d}.
$$
Moreover, by using the similar arguments as in e.{\,}g. \cite{AbCaTo,Tr,CaTo},
it can be shown that the map $a\mapsto e^{-i\scal {ry+tz} \zeta }a$ is continuous on
$(\maclS _{s_1,\sigma_2,\sigma_3}^{\sigma_1, s_2,s_3} )'(\rr {3d})$.

The following corollary is a consequence of
\cite[Theorem 4.6]{CaTo} and Proposition \ref{Prop:CalculiTransfer}.

\par

\begin{cor}\label{Somega}
Let $s_j,\sigma _j>0$ be such that $s_j +\sigma_j \ge 1$, $(s_j,\sigma_j ) \neq (\frac 12,\frac 12)$,
$j=1,2,3$ and $r_j,t_j\in[0,1]$ be such that $r_j+ t_j\leq 1$, $j=1,2$.
If  $a,b\in (\Sigma _{s_1,\sigma_2, \sigma_3}^{\sigma_1,s_2,s_3})'(\rr {3d})$,
are such that $\op _{r_1,t_1}(a)=\op _{r_2,t_2}(b)$
Then
\begin{alignat*}{3}
a &\in \Gamma ^{\sigma_1,s_2,s_3;0}_{(\omega)}(\rr {3d})&\qquad
&\Leftrightarrow &
\qquad b &\in \Gamma ^{\sigma_1,s_2,s_3;0}_{(\omega)}(\rr {3d})
\intertext{and}
a &\in \Gamma^{\sigma_1,s_2,s_3}_{(\omega)}(\rr {3d})&\qquad
&\Leftrightarrow &
\qquad b &\in \Gamma^{\sigma_1,s_2,s_3}_{(\omega)}(\rr {3d}),
\end{alignat*}
for any given  $\omega \in \mascP _{E}(\rr {3d})$.
\end{cor}

\par

\par

Passages between different kinds of pseudo-differential calculi have been
considered before  \cite{Ho1,Tr}.
On the other hand, for the bilinear pseudo-differential calculi,
it seems that the representation
$a\mapsto \op _{r,t}(a)$ for $(r,t)\in [0,1]\times[0,1]$ such that $r+t\leq 1$,
has not been considered so far.

\par

\subsection{Gevrey-type symbol classes characterizations}  \label{subsec2.2}.
Our first result concerns the Roumieu case of symbols in
$\Gamma ^{ \sigma,s,s}_{(\omega )} (\rr {3d})$. It
can be deduced from \cite[Proposition 4.3]{CaTo}, see also \cite[Proposition 2.4]{AbCaTo}.
For the sake of completeness, we give the proof.

\par

\begin{prop}\label{Prop:CharGammaSTFT}
Let  $s_j ,\sigma _j >0$,  $j=1,2,3$, be such that
the conditions in \eqref{eq:conditions} hold, let
$\omega \in \mascP _{s_1,\sigma_2,\sigma_3} ^0(\rr {3d})$
and let $a$ be a Gelfand-Shilov distribution on $\rr {3d}$.

Then the following conditions are equivalent:	
\begin{enumerate}
\item $ a \in \Gamma ^{\sigma_1 ,s_2, s_3}_{(\omega )} (\rr {3d})$, that is,
$a\in C^\infty (\rr {3d})$ and
$$
|\partial_x ^\alpha \partial _\xi ^\beta \partial _\eta^\gamma a(x, \xi, \eta)|
\lesssim
h ^{|\alpha +\beta+ \gamma|} \alpha !^{\sigma_1} \beta!^{s_2} \gamma!^{s_3} \omega (x, \xi, \eta), \;\;\;
x, \xi, \eta \in \rr {d},
$$
for every $ \alpha,\beta,\gamma\in \nn{d}$ and some $h >0$;

\vrum

\item For every  $\phi \in \maclS _{s_1,\sigma_2,\sigma_3}^{\sigma_1 ,s_2, s_3}(\rr {3d})\setminus 0$,
there exist  constants  $h,R>0$ such that for every $\alpha,\beta,\gamma\in \nn{d}$,
\begin{multline}\label{stftconda}
|\partial _x^\alpha \partial _\xi ^\beta \partial _\eta ^\gamma
\prn{e^{i(\scal x\zeta +\scal y\xi +\scal z\eta )}
	V_\phi a(x,\xi, \eta, \zeta, y, z)}|
\\[1ex]
\lesssim
h^{|\alpha+\beta+\gamma|} \alpha !^{\sigma_1} \beta!^{s_2} \gamma!^{s_3}
\omega (x, \xi, \eta)e^{-R (|\zeta|^{\frac 1{\sigma_1}}+
	|y |^{\frac 1{s_2}}+|z |^{\frac 1{s_3}})},
\end{multline}
$x,\xi, \eta, \zeta, y, z\in \rr d$.

\vrum

\item For every $\phi \in \maclS _{s_1,\sigma_2,\sigma_3}^{\sigma_1 ,s_2, s_3}(\rr {3d})\setminus 0$, there exist a constant  $R >0$ such that
\begin{equation}\label{stftcondaA}
|	V_\phi a(x,\xi, \eta, \zeta, y, z)|
\lesssim
\omega (x,\xi,\eta)e^{-R (|\zeta|^{\frac 1{\sigma_1}}+
	|y |^{\frac 1{s_2}}+|z |^{\frac 1{s_3}})}, \;\;\;
x,\xi, \eta, \zeta, y, z\in \rr d.
\end{equation}
\end{enumerate}		
\end{prop}

\par

\begin{proof}
That (2) implies (3) is immediate, since \eqref{stftconda} is equal to \eqref{stftcondaA} when $\alpha =\beta =\gamma = 0.$

Let $X=(x,\xi,\eta),\ Y=(x_1,\xi_1,\eta_1 ), Z=(\zeta,y,z)\in \rr{3d}$,
and set
$$
F_a(X,Y) =  a (X+Y)\phi(Y) =
a(x+x_1,\xi +\xi_1, \eta+ \eta_1 )\phi(x_1, \xi_1, \eta_1).
$$
Assume that (1) holds true. By the Leibniz rule,
\eqref{eq:omegaEst} and Proposition \ref{prop:GScharac} we obtain
\begin{multline*}
|\partial _x^\alpha \partial _\xi ^\beta \partial _\eta^\gamma
F_a(x,\xi, \eta, x_1, \xi_1, \eta_1)|
\\[1ex]
\lesssim
h^{|\alpha + \beta +\gamma|}\alpha !^{\sigma_1} \beta!^{s_2} \gamma!^{s_3}
\omega(x,\xi,\eta )e^{-R(|x_1|^{\frac 1{s_1}}+
|\xi_1 |^{\frac 1{\sigma_2}}+|\eta_1 |^{\frac 1{\sigma_3}})},
\end{multline*}
$x,\xi, \eta, x_1, \xi_1, \eta_1  \in \rr{d} $, for some constants $h,R >0$.

It follows that the set
$$
\Big \{ G_{a, h, X}(Y) \;\; | \;\;
 G_{a, h,x,\xi, \eta }( x_1, \xi_1, \eta_1 ) =
 \frac{\partial _x^\alpha \partial _\xi ^\beta
	\partial _\eta^\gamma F_a(x,\xi, \eta, x_1, \xi_1, \eta_1 )}
{h^{|\alpha +\beta +\gamma|} \alpha !^{\sigma_1} \beta!^{s_2} \gamma!^{s_3} \omega(x,\xi, \eta )} \Big \}
$$
is bounded in $ \maclS _{s_1,\sigma_2,\sigma_3}^{\sigma_1 ,s_2, s_3}(\mathbf{R}^{3d})$.
If $\mascF _2 F_a$ denotes  the partial Fourier transform of
$F_a(X, Y)$ with respect to the $Y$-variable, then we get
\begin{multline*}
|\partial _x^\alpha \partial _\xi ^\beta \partial_\eta^\gamma
(\mascF _2 F_a)(x,\xi, \eta, \zeta, y,z)|
\\[1ex]
\lesssim
h^{|\alpha +\beta +\gamma |} \alpha !^{\sigma_1} \beta!^{s_2} \gamma!^{s_3}
\omega(x,\xi, \eta )e^{-R(|y|^{\frac 1{\sigma_1}}+
	|z|^{\frac 1{s_2}} + |\zeta|^{\frac 1{s_3}})},
\end{multline*}
$x,\xi, \eta, x_1, \xi_1, \eta_1  \in \rr{d} $, for some constants  $h,R>0$. This, together with the Leibnitz rule
applied to
$ \partial _x^\alpha \partial _\xi ^\beta \partial _\eta ^\gamma
\prn{e^{i(\scal x\zeta +\scal y\xi +\scal z\eta )}
	V_\phi a(x,\xi, \eta, \zeta, y, z)}$ gives (2).

\par

Assume now that (3) holds. By the inversion formula we get
\begin{equation}\label{Eq:STFTInversionFormula}
a(X) = \frac{ (2\pi)^{-\frac {3d}2}}{
\nm \phi{L^2}^{2} } \iint V_\phi a(Y,Z)
\phi (X-Y)e^{i\scal {X}{Z}}\,dYdZ, \;\;\; X \in \rr {3d},
\end{equation}
in the weak sense. Since $\phi\in  \maclS _{s_1,\sigma_2,\sigma_3}^{\sigma_1 ,s_2, s_3}(\rr{3d})$
we notice that
\begin{align*}
(X,Y,Z ) &\mapsto  V_\phi a(Y,Z)\phi (X-Y)e^{i\scal {X}{Z}}
\intertext{is a smooth map, and}
(Y,Z) &\mapsto Z^{\bs{\alpha}}
V_\phi a(Y,Z )\partial ^{\bs{\beta}} \phi (X-Y)e^{i\scal XZ}
\end{align*}
is an integrable function for every $X\in \rr{3d}$,
$\bs{\alpha}=(\alpha_1,\alpha_2,\alpha_3)\in \nn{3d}$,
and $\bs{\beta}=(\beta_1,\beta_2,\beta_3)\in \nn{3d}$
in view of (3). Hence the derivatives of $a$ in \eqref{Eq:STFTInversionFormula} satisfy the following estimates:
\begin{multline*}
|\partial ^{\boldsymbol{\alpha}} a(X)|
\le
\sum _{\bs{\beta} \le \bs{\alpha}} {\bs{\alpha} \choose \bs{\beta}}
\iint |Z ^{\bs{\beta}} V_\phi a(Y,Z )
(\partial ^{\bs{\alpha} -\bs{\beta} }\phi )(X-Y)|\, dYdZ
\\[1ex]
\lesssim
\sum _{\bs{\beta} \le \bs{\alpha}} {\bs{\alpha} \choose \bs{\beta}}
\iint |Z^{\bs{\beta}}
\omega (Y)e^{-R (|\zeta|^{\frac 1{\sigma_1}}+
|y |^{\frac 1{s_2}}+|z |^{\frac 1{s_3}})}
(\partial ^{\bs{\alpha} -\bs{\beta} }\phi )(X-Y)|
\, dYdZ
\\[1ex]
\lesssim
\sum _{\bs{\beta} \le \bs{\alpha}} {\bs{\alpha} \choose \bs{\beta}}
h_2 ^{|\bs{\alpha}- \bs{\beta}|} (\alpha _1- \beta _1 )!^{\sigma_1}
(\alpha _2- \beta _2 )!^{s_2} (\alpha _3- \beta _3 )!^{s_3}
\\[1ex]
\times \iint |Z^{\bs{\beta}}|
e^{-R  (|\zeta|^{\frac 1{\sigma_1}}+
|y |^{\frac 1{s_2}}+|z |^{\frac 1{s_3}})}
\omega (Y)
e^{-h_1(|x-x_1|^{\frac 1{s_1}}+|\xi-\xi _1|^{\frac 1{\sigma_2}}+
|\eta-\eta _1|^{\frac 1{\sigma_3}})}\, dYdZ,
\end{multline*}
$X\in  \rr{3d}$, for some constants $h_1, h_2>0$, and we used  Lemma \ref{GSlemma}.
For any
$\bs{\beta} \in \nn{3d}$,
$\sigma, s_j >0$ such that $s_j + \sigma\ge 1$, $j=1,2$, and
$h_2, R >0$, it holds
$$
|\zeta ^{\beta_1 }y^{\beta_2}z^{\beta_3}
e^{-R (|\zeta|^{\frac 1\sigma}+|y |^{\frac 1s}+|z |^{\frac 1s})}|
\lesssim h_2^{|\bs{\beta}|}\beta_1 !^{\sigma_1}
\beta _2 !^{s_2}\beta _3 !^{s_3}
e^{-\frac R2 (|\zeta|^{\frac 1{\sigma_1}}+
	|y |^{\frac 1{s_2}}+|z |^{\frac 1{s_3}})},
$$
$ \zeta,y,z \in \rr d $, so that
\begin{multline}\label{Eq:GSTypeEst1}
|\partial ^{\bs{\alpha}} a(X)|
\\[1ex]
\lesssim h_2 ^{|\bs{\alpha} |}
\sum _{\bs{\beta} \le \bs{\alpha}}
{\bs{\alpha} \choose \bs{\beta}}
(\beta _1 !(\alpha _1-\beta _1)!)^{\sigma_1}
(\beta _2 !(\alpha _2-\beta _2)!)^{s_2}
(\beta _3 !(\alpha _3-\beta _3)!)^{s_3}
\\[1ex]
\times \iint e^{-\frac R2 (|\zeta|^{\frac 1{\sigma_1}}+
	|y |^{\frac 1{s_2}}+|z |^{\frac 1{s_3}})}
\omega (Y)
e^{-h_1(|x-x_1|^{\frac 1{s_1}}+|\xi-\xi _1|^{\frac 1{\sigma_2}}+
	|\eta-\eta _1|^{\frac 1{\sigma_3}})} \, dYdZ
\\[1ex]
\lesssim (2h_2)^{|\bs{\alpha} |}
\alpha _1!^{\sigma_1} \alpha _2!^{s_2} \alpha _3!^{s_3}
\\[1ex]
\times
\int \omega(X+(Y-X))
e^{-h_1(|x-x_1|^{\frac 1{s_1}}+|\xi-\xi _1|^{\frac 1{\sigma_2}}+
	|\eta-\eta _1|^{\frac 1{\sigma_3}})} \, dY, \;\;\; X \in  \rr{3d}.
\end{multline}

\par

Since $\omega \in \mascP _{s_1,\sigma_2, \sigma_3 } ^0$,
that is \eqref{eq:omegaEst} holds for every $r>0$, by choosing
$r\in(0,{h_1}/{2})$,
from \eqref{Eq:GSTypeEst1} it follows that
\begin{equation*}
|\partial ^{\bs{\alpha}} a(X)|
\lesssim (2h_2)^{|\bs{\alpha}|}
\alpha _1!^{\sigma_1} \alpha _2!^{s_2} \alpha _3!^{s_3}
\omega (X), \;\;\; X \in  \rr{3d},
\end{equation*}
for some constant $h_2>0$ (and we conclude that \eqref{Eq:STFTInversionFormula} holds also in the pointwise sense).
Therefore (3) implies (1) and the result follows.
\end{proof}

\par

The Beurling case follows by similar arguments as in the proof of
Proposition \ref{Prop:CharGammaSTFT}.

\par

\begin{prop}\label{prop:GammaOmegaChar}
Let  $s_j ,\sigma _j >0$,  $(s_j,\sigma_j )\neq (\frac 12,\frac 12)$,  $j=1,2,3$, and let
the conditions in \eqref{eq:conditions} hold. Also,
let $\omega \in \mascP _{s_1,\sigma_2, \sigma_3 } (\rr {3d})$
and let $a$ be a Gelfand-Shilov distribution on $\rr {3d}$.
Then the following conditions are equivalent:	
\begin{enumerate}
\item $ a\ in \in \Gamma ^{\sigma_1 ,s_2, s_3; 0}_{(\omega )} (\rr {3d})$, i.e.
$a\in C^\infty (\rr {3d})$ and
$$
|\partial_x ^\alpha \partial _\xi ^\beta \partial _\eta^\gamma a(x, \xi, \eta)|
\lesssim
h ^{|\alpha +\beta+ \delta|}\alpha !^{\sigma_1} \beta!^{s_2} \gamma!^{s_3} \omega (x, \xi, \eta), \;\;\;
x, \xi, \eta \in \rr {d},
$$
for every $ \alpha,\beta,\gamma\in \nn{d}$ and for every $h >0$;

\vrum

\item For every
$\phi \in \Sigma _{s_1,\sigma_2,\sigma_3}^{\sigma_1 ,s_2, s_3} (\rr {3d}) \setminus 0$, $h,R>0$ and $\alpha,\beta,\gamma\in \nn{d}$, it holds
\begin{multline*}
|\partial _x ^\alpha \partial _\xi ^\beta \partial _\eta ^\gamma
\prn{e^{i(\scal x\zeta +\scal y\xi +\scal z\eta )}
	V_\phi a(x,\xi, \eta, \zeta, y, z)}|
\\[1ex]
\lesssim
h^{|\alpha+\beta+\gamma|} \alpha !^{\sigma_1} \beta!^{s_2} \gamma!^{s_3}
\omega (x, \xi, \eta)e^{-R (|\zeta|^{\frac 1{\sigma_1}}+	|y |^{\frac 1{s_2}}+|z |^{\frac 1{s_3}})},
\end{multline*}
$x,\xi, \eta, \zeta, y, z\in \rr d$.
\vrum

\item  For every  $\phi \in \Sigma _{s_1,\sigma_2,\sigma_3}^{\sigma_1 ,s_2, s_3} (\rr {3d}) \setminus 0$, and $R >0$ it holds
$$
|V_\phi a(x,\xi, \eta, \zeta, y, z)|
\lesssim
\omega (x,\xi,\eta)e^{-R (|\zeta|^{\frac 1{\sigma_1}}+
	|y |^{\frac 1{s_2}}+|z |^{\frac 1{s_3}})}, \;\;\;
x,\xi, \eta, \zeta, y, z\in \rr d.
$$
\end{enumerate}		
\end{prop}

In the next result we consider the Beurling case.
It gives a description of  the symbol
class $\Gamma ^{\sigma_1,s_2,s_3;0}_{(\omega)} (\rr {3d})$
in terms of modulation spaces $M^{\infty,q}_{(1/\omega _R)} (\rr {3d})$
for $q \in[1,\infty]$ and $\omega_R$ defined in
\eqref{eq:omegaR} below. To prove Proposition \ref{Prop:GammaModIdent*}
we follow arguments analogous to those used in the proofs of \cite[Proposition 3.5]{AbCaTo} and \cite[Proposition 4.4]{CaTo}.

\par

\begin{prop}\label{Prop:GammaModIdent*}
Let $R>0$, $q\in [1,\infty ]$, $s_j,\sigma _j>0$,  $(s_j,\sigma_j ) \neq (\frac 12,\frac 12)$, $j=1,2,3$,
and let the conditions in \eqref{eq:conditions} hold.
Also, let $\phi, \phi_0 \in \Sigma _{s_1,\sigma_2, \sigma_3}^{\sigma_1,s_2 ,s_3} (\rr {3d})\setminus 0$,
$\omega \in \mascP _{s_1,\sigma_2, \sigma_3}(\mathbf{R}^{3d})$,
and let
\begin{equation}\label{eq:omegaR}
\omega _R(x,\xi, \eta, \zeta,y,z) =
\omega(x,\xi,\eta) e^{-R(|\zeta|^{\frac 1{\sigma_1}}+|y|^{\frac 1{s_2}}+|z|^{\frac 1{s_3}})}.
\end{equation}
Then
\begin{equation}\label{iden*}
\Gamma ^{\sigma_1,s_2,s_3;0} _{(\omega)} (\rr {3d})=
\bigcap _{R>0}\sets {a\in
	(\Sigma _{s_1,\sigma_2, \sigma_3}^{\sigma_1,s_2,s_3})'
	(\rr {3d})}{\nm {\omega	_R^{-1}V_\phi a}{L^{\infty ,q}
		(\rr {3d}\times\rr{3d})} <\infty }.
\end{equation}
\end{prop}

\par

\begin{proof}
When $q =\infty$, \eqref{iden*} becomes
$$
\Gamma ^{\sigma_1,s_2,s_3;0}_{(\omega)}=
\bigcap_{R>0} M^\infty_{(1/\omega _R)}
(\mathbf{R}^{3d}),
$$
which is a straightforward consequence of Proposition \ref{prop:GammaOmegaChar}.
Therefore it is enough to prove that
\begin{multline*}
\bigcap_{R>0} M^\infty_{(1/\omega _R)}
(\mathbf{R}^{3d}) 
\\[1ex]
=\bigcap _{R>0}\sets {a\in
	(\Sigma _{s_1,\sigma_2, \sigma_3}^{\sigma_1,s_2,s_3})'
	(\rr {3d})}{\nm {\omega	_R^{-1}V_\phi a}{L^{\infty ,q}
		(\rr {3d}\times\rr{3d})} <\infty }.
\end{multline*}
\par

Put
\begin{gather*}
V_{0,a}(X,Y)=|(V_{\phi _0}a)(x,\xi,\eta, \zeta, y, z)|,
\quad
V_{a}(X,Y)=|(V_{\phi}a)(x,\xi,\eta, \zeta, y, z)|
\\[1ex]
\text{and}\quad
G(x,\xi,\eta, \zeta, y, z)=|(V_\phi {\phi _0})(x,\xi,\eta, \zeta, y, z)|,
\end{gather*}
where $X=(x,\xi, \eta) \in \mathbf{R}^{3d}$ and
$Y=(\zeta, y, z)  \in \mathbf{R}^{3d}$. By Proposition \ref{Prop:STFTGelfand2}
we have
\begin{equation}\label{GEst1}
0\le G(x,\xi,\eta, \zeta, y, z) \lesssim
e^{-R(|x|^{\frac 1{s_1}}+|\xi|^{\frac 1{\sigma_2}}+|\eta|^{\frac 1{\sigma_3}}
+|\zeta|^{\frac 1{\sigma_1}}+|y|^{\frac 1{s_2}}+|z|^{\frac 1{s_3}})},
\end{equation}
$x,\xi,\eta, \zeta, y, z \in \rr d$, for every $R>0 .$

\par

From \cite[Lemma 11.3.3]{Gro} (when extended to the duality between Gelfand-Shilov spaces and their dual spaces of distributions)
it follows that
$V_a \lesssim V_{0,a} \ast G$, so we obtain
\begin{equation}\label{pippo}
(\omega_R^{-1} \cdot V_a)(X,Y)
\lesssim
\left((\omega^{-1}_{cR} \cdot V_{0,a})*G_1\right)(X,Y),
, \;\;\;
X,Y \in \mathbf{R}^{3d},
\end{equation}
for some $G_1$ which satisfies \eqref{GEst1},
and for a constant $c>0$ independent of $R$.
By applying the $L^\infty$-norm on the both sides of \eqref{pippo} we obtain
\begin{multline*}
\nm {\omega _R^{-1}V_a}{L^{\infty}(\rr {6d})} =
\sup _Y \sup _X \abs{\omega_R^{-1}V_a(X,Y)}
\\[1ex]
 \lesssim \sup _Y \sup _X \abs{\omega_{cR}^{-1}V_{0,a}*G_{1}(X,Y)}
\\[1ex]
\lesssim \sup_{Y}\left(\iint
\big ( \sup _X(\omega^{-1}_{cR} \cdot V_{0,a})(X , Y-Y_1)\big )
G_1(X_1, Y_1)\,  dX_1 dY_1\right)
\\[1ex]
\le
\nm {\omega^{-1}_{cR} \cdot V_{0,a}}{L^{\infty ,q}}
\nm {G_1}{L^{1,q'}}
\asymp \nm {\omega^{-1}_{cR} \cdot V_{0,a}}{L^{\infty ,q}},
\end{multline*}
wherefrom
\begin{multline*}
\bigcap _{R>0}\sets {a\in
	(\Sigma _{s_1,\sigma_2, \sigma_3}^{\sigma_1,s_2,s_3})'
	(\rr {3d})}{\nm {\omega	_R^{-1}V_\phi a}{L^{\infty ,q}
		(\rr {3d}\times\rr{3d})} <\infty } \\[1ex]
\subset \bigcap_{R>0} M^\infty_{(1/\omega _R)} (\rr {3d}).
\end{multline*}

\par

For the opposite inclusion we put
$K_j=\omega^{-1}_{jcR} \cdot V_{0,a}$, $j=1,2$.
By \eqref{pippo} and Minkowski's inequality we have
\begin{multline*}
\| \omega^{-1}_{R} \cdot V_{a} \| _{L^{\infty,q}}^q
\lesssim
\| \left((
K_1*G\right)\| _{L^{\infty,q}}^q
\\[1ex]
\lesssim
\int \sup _X\left(\iint K _1 (X-X_1, Y-Y_1)G(X_1, Y_1)
\, dX_1 dY_1\right)^q\, dY
\\[1ex]
\lesssim  \int
 \left(\iint  \sup \left(K _2(\cdo , Y-Y_1)\right)G(X_1, Y_1) \right.
\\[1ex]
\left. \times e^{
-cR(|\zeta-\zeta_1|^{\frac 1{\sigma_1}}+|y-y_1|^{\frac 1{s_2}}+
	|z-z_1|^{\frac 1{s_3}})}
\, dX_1 dY_1 \right)^q dY
\\[1ex]
\lesssim \| K _2 \|_{L^\infty}^q
\int \left(\iint G(X_1, Y_1)
e^{-cR(|\zeta-\zeta_1|^{\frac 1{\sigma_1}}+|y-y_1|^{\frac 1{s_2}}+
	|z-z_1|^{\frac 1{s_3}})}\, dX_1 dY_1 \right)^q\, dY
\\[1ex]
\lesssim \|K _2\|_{L^\infty}^q\equiv\| \omega^{-1}_{2cR} \cdot V_{0,a} \|_{L^\infty}^q.
\end{multline*}
Finally, by interchanging the roles of $\phi$ and $\phi _0$ we get
$$
\| \omega^{-1}_{R} \cdot V_{0,a} \| _{L^{\infty,q}} \lesssim
\| \omega^{-1}_{2cR} \cdot V_{a} \|_{L^\infty},
$$
i.e.
\begin{multline*}
\bigcap_{R>0} M^\infty_{(1/\omega _R)}
(\rr {3d}) \\[1ex] 
\subset
\bigcap _{R>0}\sets {a\in
	(\Sigma _{s_1,\sigma_2, \sigma_3}^{\sigma_1,s_2,s_3})'
	(\rr {3d})}{\nm {\omega	_R^{-1}V_\phi a}{L^{\infty ,q}
		(\rr {3d}\times\rr{3d})} <\infty },
\end{multline*}
and the result follows.	
\end{proof}

\par

\par

We leave for the reader to write down Proposition \ref{Prop:GammaModIdent*} when the (Roumieu case) symbol class $\Gamma^{ \sigma_1,s_2, s_3}_{(\omega)}(\rr {3d}) $
is considered instead.

\par

In \cite[Theorem 4.1]{CaTo},
it is shown that if $A$ is  a $d\times d$-matrix  with real entries,
then the operator
$e^{i\scal{AD_\xi}{D_x}}$ is a homeomorphism between certain classes of symbols.
We proceed with an analogous result in the context of the symbol class $\Gamma _{(\omega)} ^{\sigma_1,s_2,s_3;0}(\rr {3d})$.

\par

\begin{thm}\label{thm:CalculiTransfbis}
Let  $s_j ,\sigma _j > 0$ be such that the conditions in \eqref{eq:conditions} hold, and
 let $r,t\in [0,1]$ be such that $r+t \leq 1.$

If $\omega \in \mascP _{s_1,\sigma_2,\sigma_3}^0(\rr {3d})$,
then $a\in \Gamma _{(\omega)} ^{\sigma_1,s_2,s_3}(\rr {3d})$
if and only if 
$$e^{-i\scal {rD_\xi+tD_\eta}{D_x}}
a\in \Gamma _{(\omega)} ^{\sigma_1,s_2,s_3}(\rr {3d}).	$$

If $\omega \in \mascP _{s_1,\sigma_2,\sigma_3}(\rr {3d})$ instead, and if, in addition to
\eqref{eq:conditions},  $(s_j,\sigma_j ) \neq (\frac 12,\frac 12)$, $j = 1,2,3$,
then $a\in \Gamma _{(\omega)} ^{\sigma_1,s_2,s_3;0}(\rr {3d})$
if and only if 
$$
e^{-i\scal {rD_\xi+tD_\eta}{D_x}} a\in \Gamma _{(\omega)} ^{\sigma_1,s_2,s_3;0}(\rr {3d}).
$$
\end{thm}

\par

\begin{proof} We give the proof for the Beurling  case, and the Roumieu case is left for the reader.

We will use the result of Proposition \ref{Prop:GammaModIdent*}.
Therefore we fix a window function  $\phi \in \Sigma _{s_1,\sigma_2,\sigma_3}^{\sigma_1, s_2,s_3}(\rr {3d})$ and
let $\phi _{r,t} =e^{-i\scal {rD_\xi+tD_\eta}{D_x}}\phi$. Then,
in view of Theorem \ref{Thm:CalculiTransf} (2),
$\phi _{r,t}$ belongs to $\Sigma _{s_1,\sigma_2,\sigma_3}^{\sigma_1, s_2,s_3} (\rr {3d})$.

\par

By similar arguments as in the proof of Lemma \ref{lem:FTInvOpr}, we get
\begin{multline}\label{eq:ModSTFTequa}
|(V_{\phi _{r,t}} (e^{-i\scal {rD_\xi+tD_\eta}{D_x}}a))(x,\xi ,\eta,\zeta ,y,z)|
\\
=
|(V_\phi a)(x+ry+tz,\xi +r\zeta,\eta+t\zeta ,\zeta ,y,z)|,
\end{multline}
$ x,\xi ,\eta,\zeta ,y,z \in \rr d$.
Then using \eqref{eq:ModSTFTequa} and  a change of variables argument,
we get
$$
\nm {\omega _{0,0;R}^{-1}V_{\phi} a}{L^{p,q}} =
\nm {\omega _{r,t;R}^{-1}V_{\phi _{r,t}}
(e^{-i\scal {rD_\xi+tD_\eta}{D_x}}a)}{L^{p,q}},
$$
where
$$
\omega _{r,t;R}(x,\xi ,\eta,\zeta ,y,z) = \omega (x+ry+tz,\xi +r\zeta,\eta+t\zeta )
e^{-R(|y|^{\frac 1{s_2}}+|z|^{\frac 1{s_3}}+|\zeta |^{\frac 1{\sigma_1}})},
$$
and $p,q \in [1,\infty]$.

Hence Proposition \ref{Prop:GammaModIdent*}, and the fact that there exists a constant $c>0$ such that
$$
\omega _{0,0;R+c}\lesssim \omega _{r,t;R}\lesssim \omega _{0,0;R-c},
$$
 give
\begin{multline*}
a\in\Gamma _{(\omega)} ^{\sigma_1,s_2,s_3;0}(\rr {3d})
\quad \Leftrightarrow \quad
\nm {\omega _{0,0;R}^{-1}V_\phi a}{L^\infty}<\infty \;
\text{for every}
\ R>0
\\[1ex]
\Leftrightarrow \quad
\nm {\omega _{r,t;R}^{-1}V_{\phi _{r,t}}
(e^{-i\scal {rD_\xi+tD_\eta}{D_x}}a)}{L^\infty} <\infty \;
\text{for every} \ R>0
\\[1ex]
\Leftrightarrow \;
\nm {\omega _{0,0;R}^{-1}V_{\phi _{r,t}}
(e^{-i\scal {rD_\xi+tD_\eta}{D_x}}a)}{L^\infty}
<\infty \;
\text{for every} \ R>0
\\[1ex]
\Leftrightarrow \quad
e^{-i\scal {rD_\xi+tD_\eta}{D_x}}a \in
\Gamma _{(\omega)} ^{\sigma_1,s_2,s_3;0} (\rr {3d}),
\end{multline*}
and the result follows.

\end{proof}

\section{Continuity of bilinear pseudo-differential operators
with symbols of Gevrey-regularity and infinite order} \label{sec3}

We first discuss the continuity of bilinear operators in
$\op (\Gamma  ^{\sigma_1,s_2,s_3}_{(\omega)})$
and $\op (\Gamma ^{\sigma_1,s_2,s_3;0}_{(\omega)})$
when acting on products modulation spaces.
In particular, Theorem \ref{thm:p3.2}  can be considered as
an extension of \cite[Theorem 3.2]{To14} to
bilinear operators and a more general class of weights.

\par
%

\begin{thm}\label{thm:p3.2}
Let $s_j,\sigma _j>0$ be such that the conditions in \eqref{eq:conditions} hold.
Also, let
$v _1\in \mascP _{s_1} ^0(\rr d)$,
$v _j\in\mascP _{\sigma_j} ^0(\rr {d})$, $j=2,3$,
$\omega _0,\omega \in\mascP _{s_1,\sigma_2,\sigma_3}^0(\rr {3d})$,
and let $\omega _0$ be $\otimes _{j=1}^3 v_j$-moderate.
Furthermore, let $r,t\in [0,1]$ such that $r+t\leq 1$, and let
$p,q \in [1,\infty]$.
If  $a\in \Gamma _{(\omega _0)}^{\sigma_1,s_2,s_3}(\rr {3d})$,
then there exists $R>0$ such that
$\op _{r,t}(a)$ is continuous from
$M_{(\omega _0\omega)}^{p,q} (\rr {d}) \times M _{(1/\omega_R)}^{\infty,\infty}  (\rr {d}) $
to $M _{(\omega)}^{p,q}  (\rr {d}) $,
where
$$
\omega _R(x,\xi,\eta)=e^{-R(|x|^{\frac 1{s_1}} +|\xi|^{\frac 1{\sigma_2}}+|\eta|^{\frac 1{\sigma_3}})},
\;\;\; x,\xi,\eta \in \rr d.
$$
\end{thm}
\par

\par

\begin{rem}\label{rem:BiltoLin}
We will use estimates similar to those obtained in the  proof of \cite[Theorem 6.1]{BeMaNaTo}. We observe that out arguments are anyway different since, in view of the fact that we employ Gevrey tpye symbols, we cannot rely on standard localization techniques.
The idea is that for a fixed function $g$ in appropriate space of test functions, $a\in \Gamma _{(\omega )}^{\sigma_1,s_2,s_3}(\rr {3d})$, and $r=t=0$, the operator
$\op_{0,0}(a)(\cdo,g)\equiv T_a(\cdo, g)$
can be regarded as a linear pseudo-differential operator
(with symbol depending on $g$), that is,
\[
\op_{0,0}(a)(f,g)(x)=(2\pi)^{-\frac{d}{2}}
\int e^{i \scal x\xi} a_g(x,\xi)\widehat{f}(\xi)\,d\xi,
\]
where
\begin{equation}\label{eq:a_gDef}
a_g(x,\xi)=(2\pi)^{-\frac{d}{2}}
\int e^{i \scal x\eta} a(x,\xi,\eta)\widehat{g}(\eta)\,d\eta.
\end{equation}
If the symbol $a_g$ belongs to $\Gamma ^{\sigma_1,s_2}_{(\omega_0)} ( \rr {2d})$, then
the continuity  of $\op_{0,0}(a)(\cdo,g)$ from
$M^{p.q}_{(\omega_0\omega)}$ to $M^{p.q}_{(\omega)}$  follows  by
Proposition \ref{thm:anisCntp}.
\end{rem}

\par

\begin{lemma}\label{lemma:a_gClass}
Let $s_j,\sigma _j>0$ be such that the conditions in \eqref{eq:conditions} hold.
Also, let
$v _1\in \mascP _{s_1} ^0(\rr d)$,
$v _j\in\mascP _{\sigma_j} ^0(\rr {d})$, $j=2,3$,
$\omega _0,\omega \in\mascP _{s_1,\sigma_2,\sigma_3}^0(\rr {3d})$,
and let $\omega _0$ be $\otimes _{j=1}^3 v_j$-moderate.

If $g\in \maclS _{s_1}^{\sigma_1}(\rr d)$
and $a\in \Gamma _{(\omega_0 )}^{\sigma_1,s_2,s_3}(\rr {3d})$,
then the symbol $a_g$ given by \eqref{eq:a_gDef} belongs
to $\Gamma _{(\omega )}^{\sigma_1,s_2}(\rr {2d})$, where
$\omega(x,\xi)
\equiv \omega _0(x,\xi,0) \in\mascP _{s_1,\sigma_2}^0(\rr {2d})$.	
\end{lemma}

\par

\begin{proof}
By \eqref{eq:a_gDef} it follows that $a_g$ is a smooth function.
Indeed,
$$
(x,\xi,\eta)\mapsto
e^{i \scal x\eta} a(x,\xi,\eta)\widehat{g}(\eta)
$$
is a smooth mapping and
$$
\eta\mapsto
\eta^\gamma e^{i \scal x\eta} \partial _x^\alpha
\partial _\xi^\beta a(x,\xi,\eta)\widehat{g}(\eta)
$$
is an integrable function for every $x,\xi, \alpha,\beta$
and $\gamma$.

\par

Since $\widehat{g} \in \maclS  ^{s_1} _{\sigma_1}(\rr d)$ (cf. Proposition
\ref{propBroadGSSpaceChar}) and since
$\omega _0\in \mascP _{s_1,\sigma_2,\sigma_3}^0 (\rr {3d})$, it follows that
\begin{multline*}
\abs{\partial _x^\alpha\partial _\xi^\beta a_g(x,\xi)}
\lesssim
\sum _{\gamma \leq \alpha}\binom{\alpha}{\gamma}
\int \abs{\eta^{\gamma}\partial _x^{\alpha-\gamma}
\partial _\xi^\beta a(x,\xi,\eta)\widehat{g}(\eta)}\, d\eta
\\[1ex]
\lesssim
\sum _{\gamma \leq \alpha}\binom{\alpha}{\gamma}
\int
h^{|\alpha+\beta-\gamma|}(\alpha-\gamma)!^{\sigma_1} \beta!^{s_2}
\abs{\omega _0(x,\xi,\eta)
\eta^\gamma e^{-r|\eta|^{\frac 1{\sigma_1}}}}
\, d\eta
\\[1ex]
\lesssim
\sum _{\gamma \leq \alpha}\binom{\alpha}{\gamma}
\int
h^{|\alpha+\beta-\gamma|}(\alpha-\gamma)!^{\sigma_1} \beta!^{s_2}
\omega _0(x,\xi,0)\abs{e^{r_0|\eta|^{\frac 1{\sigma_3}}}
\eta^\gamma e^{-r|\eta|^{\frac 1{\sigma_1}}}}
\, d\eta,
\end{multline*}
for every $r_0>0$, and some constants $r,h>0$.
Since $r_0$ can be chosen such that $r_0<r$,
and since
$$
\abs{\eta^\gamma e^{-(r-r_0)|\eta|^{\frac 1{\sigma_1}}}}
\lesssim h^{|\gamma|}\gamma!^{\sigma_1}
e^{-\frac{(r-r_0)}2|\eta|^{\frac 1{\sigma_1}}},
$$
we get
\begin{multline*}
\abs{\partial _x^\alpha\partial _\xi^\beta a_g(x,\xi)}
\\[1ex]
\lesssim
h^{|\alpha+\beta|} \beta!^{s_2}
\sum _{\gamma \leq \alpha}\binom{\alpha}{\gamma}
\prn{(\alpha-\gamma)!\gamma!}^{\sigma_1} \int
\omega _0(x,\xi,0)
 e^{-(r-r_0)|\eta|^{\frac 1{\sigma_1}}}
\, d\eta,
\\[1ex]
\lesssim (4h)^{|\alpha+\beta|}\alpha!^{\sigma_1} \beta!^{s_2}
\omega(x,\xi), \;\;\; x,\xi \in \rr d,
\end{multline*}
for some constant $h>0$, where $\omega(x,\xi)
\equiv \omega _0(x,\xi,0)\in \mascP _{s_1,\sigma_2}^0(\rr {2d})$.
This gives the desired result.
\end{proof}

\par

\begin{proof}[Proof of Theorem \ref{thm:p3.2}]
In view of the invariance properties for the bilinear
pseudo-differential operators given in Theorem \ref{thm:CalculiTransfbis},
we may assume $r=t=0$ without loss of generality.

\par

By Proposition \ref{thm:anisCntp}
and Lemma \ref{lemma:a_gClass}, it follows that $\op(a)(f,g)$ is a continuous mapping from
$M^{p.q}_{(\omega_0\omega)} (\rr d) \times\maclS _{s_1} ^{\sigma_1} (\rr d)$
to $M^{p.q}_{(\omega)} (\rr d)$. Now the result follows from Proposition \ref{Prop:STFTGelfand2}.
\end{proof}

\par

In a similar way, we get the  result for the Beurling case. The details are left for the reader.

\par

\begin{thm}\label{thm:p3.2bis}
Let $s_j,\sigma _j>0$,   $(s_j,\sigma_j ) \neq (\frac 12,\frac 12)$, $j=1,2,3$,
be such that the conditions in \eqref{eq:conditions} hold.
Also, let
$v _1\in \mascP _{s_1} ^0(\rr d)$,
$v _j\in\mascP _{\sigma_j} ^0(\rr {d})$, $j=2,3$,
$\omega _0,\omega \in\mascP _{s_1,\sigma_2,\sigma_3}^0(\rr {3d})$,
and let $\omega _0$ be $\otimes _{j=1}^3 v_j$-moderate.
Furthermore, let $r,t\in [0,1]$ such that $r+t\leq 1$, and let
$p,q \in [1,\infty]$.

If  $a\in \Gamma _{(\omega _0)}^{\sigma_1,s_2,s_3;0}(\rr {3d})$,
then for any $R>0$ the operator
$\op _{r,t}(a)$ is continuous from
$M_{(\omega _0\omega)}^{p,q} (\rr {d}) \times M _{(1/\omega_R)}^{\infty,\infty}  (\rr {d}) $
to $M _{(\omega)}^{p,q}  (\rr {d}) $,
where
$$
\omega _R(x,\xi,\eta)=e^{-R(|x|^{\frac 1{s_1}} +|\xi|^{\frac 1{\sigma_2}}+|\eta|^{\frac 1{\sigma_3}})},
\;\;\; x,\xi,\eta \in \rr d.
$$
\end{thm}

Finally, the characterization of  Gelfand-Shilov spaces via modulation spaces gives the following result (cf. \cite{GZ, Te6, To18}),
see also Proposition \ref{Prop:STFTGelfand2}.

\begin{thm}\label{thm:CntGSspaces}
Let there be given $s,\sigma >0$ such that $s+\sigma \ge 1$,
$v _1\in \mascP _s^0(\rr d)$, $v _j\in\mascP _\sigma^0(\rr {d})$, $j=2,3$,
and $\omega_0 \in\mascP _{s,\sigma,\sigma}^0(\rr {3d})$,
such that  $\omega_0 $ is $\otimes _{j=1}^3 v_j$-moderate.
If  $r,t\in [0,1]$, such that $r+t\leq 1$, and
$a\in \Gamma _{(\omega )}^{\sigma,s,s}(\rr {3d})$
then $\op _{r,t}(a)$ is continuous from
$\maclS _{s}^{\sigma}(\rr d)\times
\maclS _{s}^{\sigma}(\rr d)$
to $\maclS _{s}^{\sigma}(\rr d)$,
and from
$(\maclS _s^\sigma )'(\rr d)\times(\maclS _s^\sigma )'(\rr d)$
to $(\maclS _s^\sigma )'(\rr d)$.
\end{thm}

\par

\begin{proof}
In view of Theorem \ref{thm:CalculiTransfbis},
it is enough to consider the case when $r=t=0$, i.e.  $\op _{0,0}(a)$.

By Proposition \ref{Thm:theorem2} and Remark \ref{rem:BiltoLin},
it is enough to show that $a_g$ given by \eqref{eq:a_gDef} belongs to $ \Gamma _{(\omega)} ^{\sigma ,s}(\rr {2d})$
for $\omega(x,\xi) \equiv \omega _0(x,\xi,0)\in \mascP _{s,\sigma}^0(\rr {2d})$.
This follows from Lemma \ref{lemma:a_gClass}. Now the
the continuity of $\op _{r,t}(a)$ from
$\maclS _{s}^{\sigma}(\rr d)\times \maclS _{s}^{\sigma}(\rr d)$
to  $\maclS _{s}^{\sigma}(\rr d)$ follows from Theorem \ref{thm:p3.2} and Proposition \ref{Prop:STFTGelfand2}.

The continuity of $\op _{r,t}(a)$
from $(\maclS _s^\sigma )'(\rr d)\times(\maclS _s^\sigma )'(\rr d)$
to $(\maclS _s^\sigma )'(\rr d)$ follows by duality.
\end{proof}

\par

The analogous result given below follows by similar arguments and is left for the reader.

\par

\begin{thm}\label{thm:CntGSspacesbis}
Let there be given $s,\sigma >0$ such that $s+\sigma \ge 1$ and $ (s,\sigma) \neq (\frac{1}{2},\frac{1}{2}).$
Also, let $v _1\in \mascP _s^0(\rr d)$, $v _j\in\mascP _\sigma^0(\rr {d})$, $j=2,3$,
and $\omega_0 \in\mascP _{s,\sigma,\sigma}^0(\rr {3d})$,
such that  $\omega_0 $ is $\otimes _{j=1}^3 v_j$-moderate.
If  $r,t\in [0,1]$, such that $r+t\leq 1$, and
$a\in \Gamma _{(\omega )}^{\sigma,s,s;0}(\rr {3d})$
then $\op _{r,t}(a)$ is continuous from
$\Sigma _{s}^{\sigma}(\rr d)\times
\Sigma _{s}^{\sigma}(\rr d)$
to $\Sigma _{s}^{\sigma}(\rr d)$,
and from
$(\Sigma _s^\sigma )'(\rr d)\times( \Sigma _s^\sigma )'(\rr d)$
to $(\Sigma _s^\sigma )'(\rr d)$.
\end{thm}

\appendix
\section{Proofs of some auxiliary results} \label{appendix}

In this concluding section we collect the proofs of some technical results
employed above.

\begin{proof}[Proof of Lemma \ref{GSlemma}]

We give the proof for the case $ \Sigma _{s}^{\sigma}(\rr d)$ (the Beurling case)  only,
and the (simpler) case $ \maclS _{s}^{\sigma}(\rr d)$ (the Roumieu case) is left for the reader.
For simplicity (to avoid the use of inequalities related to multi-indices) we here set $d= 1$.

\par

Note that
$$
\sup_{x} |x^p f^{(q)} (x)| \leq C k^p l^{q} p!^s q!^{\sigma}, \;\;\; \forall k,l >0,
$$
is equivalent to
$$
\sup_{x} |x^p f^{(q)} (x)| \leq C h^{p + q} p!^s q!^{\sigma}, \;\;\; \forall h >0.
$$
In fact, if the former condition holds, then we may put $l=k=h$ for any given $h>0$,
so the later condition holds. If the later condition holds instead, then for any given $  k,l >0$ we may choose $h = \min\{ k,l\}$
to obtain
$$
\sup_{x} |x^p f^{(q)} (x)| \leq C \min\{ k,l\}^{p + q} p!^s q!^{\sigma} \leq  C k^p l^{q} p!^s q!^{\sigma},
$$
and the equivalence follows.

\par

Assume that  $f \in \Sigma _{s}^{\sigma}(\re)$, that is, for every $p,q \in \nn{}$,
\begin{equation} \label{prviuslov}
\sup_{x \in \re } | x^p f^{(q)} (x)| \lesssim k^p l^q p!^s q!^{\sigma}, \;\;\; \forall k,l >0.
\end{equation}
Put $F_q (x, l) = f^{(q)} (x) / (l^{q} q!^{\sigma}),$  $ x \in \re$.
Then, by \eqref{prviuslov}
$$
\sup_{x} k^{-p} p !^{-s}  | x|^{p} | F_q (x,l) | < \infty \;\;
\Rightarrow \;\;
\sup_{x} k^{-\frac{p}{s}} p !^{-1}  | x|^{\frac{p}{s}} | F_q (x,l) |^{\frac{1}{s}}  \leq C,
$$
for some constant $C>0,$ uniformly in $p,$ so that
$$
 \sup_{x}  \sum_{p \in \mathbb{N}} \left  ( \frac{(| x|k^{-1})^{1/s}}{2} \right )^{p}  p !^{-1}
 | F_q (x,l)  |^{1/s}   \leq C \sum_{p \in \mathbb{N}} \frac{1}{2^{p}}.
$$
For any given $h>0$ we may choose $ k = (s/ 2h )^s $ to obtain
\begin{multline*}
 \sup_{x}  \sum_{p \in \mathbb{N}}  ( \frac{\left ( \left(| x|\left (\frac{2h}{s}\right )^s \right )^{1/s} \right )^{p} }{2^{p} p !}
 | F_q (x,l)  |^{1/s}  \\
=
\sup_{x}  \sum_{p \in \mathbb{N}}   \frac{\left ( | x|^{1/s} \frac{h}{s}  \right )^p}{ p !}
 | F_q (x,l)  |^{1/s}
  \leq \tilde C.
\end{multline*}
Therefore
$$
e^{\frac{h}{s} |x|^{1/s}} \sup_{x} | F_q (x,l)  |^{1/s} \leq \tilde C \;\;\; \Rightarrow \;\;\;
| F_q (x,l)  |^{1/s} \leq \tilde C e^{-\frac{h}{s} |x|^{1/s}}, \;\;\; \forall h,l>0,
$$
and we conclude that
$$
\frac{|f^{(q)} (x)|}{l^{q} q!^{\sigma}} \leq \tilde{C} ^s e^{-h |x|^{1/s}}
\;\;\; \Leftrightarrow \;\;\;
|f^{(q)} (x)| \leq C l^{q} q!^{\sigma} e^{-h |x|^{1/s}}, \;\;\; \forall h,l >0,
$$
which proves the first part of Lemma.

\par

Now, assume that \eqref{GSspaces} holds, so that for $F_q (x,l) = f^{(q)} (x) / (l^{q} q!^{\sigma})$
we have
$$
|F_q (x, l)|^{1/s} e^{\frac{h}{s} |x|^{1/s}} \leq C,  \;\;\; \forall h,l>0,
$$
for some $C>0$ uniformly in $x\in \re$. Therefore
$$
\sup_{x} \sum_{p \in \mathbb{N}} \frac{1}{p !}
(\frac{h}{s})^p  |x|^{p/s} \left | F_q (x, l)  \right|^{1/s} < \infty,
$$
which implies
$$
\frac{1}{p !} \left (\frac{h}{s} \right )^p  |x|^{p/s} \left | F_q (x,l)  \right|^{1/s} < \infty
\;\;\;
\Leftrightarrow
\;\;\;
\frac{1}{p !^s} \left  (\frac{h}{s} \right )^{ps}  |x|^{p} \left | F_q (x, l)  \right| < \infty,
$$
that is
$$
|x^{p} f^{(q)} (x) | \leq C  \left  (\frac{s^s}{h^s} \right )^{p} l^{q} p !^s q!^{\sigma},
\;\;\; \forall h,l >0.
$$
Finally, for a given $k>0$ by choosing $h = s/ k^{1/s}$, we have
$ k = (s/h)^s $ and
$$
|x^{p} f^{(q)} (x) | \leq C   k^{p} l^{q} p !^s q!^{\sigma},
$$
which completes the proof.
\end{proof}

\begin{proof}[Proof of Proposition \ref{Prop:EquivWeights}]

It is enough to prove  (2) for every  $0<s_j < 1$, $j=1,\dots, k$.
In addition, we assume that $(s_0,\dots,s_k)\neq (1/2,\dots,1/2)$, and leave
the necessary modifications when $(s_0,\dots,s_k)= (1/2,\dots,1/2)$ for the reader.

Let $\phi_0 \in \Sigma _{1-s_0,\dots,1-s_k}^{s_0,\dots, s_k}(\rr{d_0 +\dots + d_k})\setminus 0$,
and $\phi =|\phi _0|^2$. Then $\phi \in _{1-s_0,\dots,1-s_k}^{s_0,\dots, s_k}(\rr{d_0 +\dots + d_k})\setminus 0$,
i.e.
\begin{multline*}
|\partial _{x} ^{\alpha_0}\partial _{\xi_1} ^{\alpha_1}
\dots \partial _{\xi_k} ^{\alpha_k} \phi (x, \xi_1,\dots,\xi_k)|
\lesssim
h^{|\alpha_0 +\alpha_1+\dots+\alpha_k|}\prod_{j=0} ^k \alpha_j !^{s_j}
\\[1ex]
\times
e^{-c(|x|^{\frac 1{1-s_0}}+|\xi_1|^{\frac 1{1-s_1}}+\dots +|\xi_k|^{\frac 1{1-s_k}})},
\end{multline*}
for every $h>0$ and $c>0$.

Now let $\omega _0=\omega *\phi$.
Then we have
\begin{multline*}
|\partial _{x} ^{\alpha_0}\partial _{\xi_1} ^{\alpha_1}
\dots \partial _{\xi_k} ^{\alpha_k}
\omega _0 (x, \xi_1,\dots,\xi_k)|
\\[1ex]
= | \int_{\rr{d_0 +\dots + d_k}} \omega (x -y, \xi_1 - \eta_1, \dots, \xi_k - \eta_k)
\\[1ex]
\times
(\partial _{x} ^{\alpha_0}\partial _{\xi_1} ^{\alpha_1}
\dots \partial _{\xi_k} ^{\alpha_k} \phi )
(y, \eta_1, \dots, \eta_k)\, dy d\eta_1 \dots d\eta_k  |
\end{multline*}
\begin{multline*}
\lesssim
h^{|\alpha_0 +\alpha_1+\dots+\alpha_k|}\prod_{j=0} ^k \alpha_j !^{s_j}
\times
 \iint\cdots\int \omega (x -y, \xi_1 - \eta_1, \dots, \xi_k - \eta_k)
\\[1ex]
\times
e^{-c(|y|^{\frac 1{1-s_0}}+|\eta_1|^{\frac 1{1-s_1}}+\dots +|\eta_k|^{\frac 1{1-s_k}})}
\, dy d\eta_1 \cdots d\eta_k
\end{multline*}
\begin{multline*}
\lesssim
h^{|\alpha_0 +\alpha_1+\dots+\alpha_k|}\prod_{j=0} ^k \alpha_j !^{s_j} \times\omega  (x, \xi_1,\dots,\xi_k)
\\[1ex]
\times
\iint\cdots\int
e^{-\frac{c}{2}(|y|^{\frac 1{1-s_0}}+|\eta_1|^{\frac 1{1-s_1}}+\dots +|\eta_k|^{\frac 1{1-s_k}})}
\, dy d\eta_1 \dots d\eta_k
\end{multline*}
$$
\asymp h^{|\alpha_0 +\alpha_1+\dots+\alpha_k|}\prod_{j=0} ^k \alpha_j !^{s_j} \times\omega  (x, \xi_1,\dots,\xi_k),
$$
where we used \eqref{eq:2Nextbis}, \eqref{eq:expEstonsubmultipicative},
and the  fact that for any $r,c>0$ and $s\in (0,1)$ it holds
$ \displaystyle e^{r|\cdot|} \lesssim e^{\frac{c}{2}|\cdot|^{1/s}}.$
This gives the first part of (2) and $ \omega _0 \lesssim \omega$, by choosing $\alpha_j = 0$, $ j =0,\dots,k$.

\par

Now, for  $ \omega _0 \gtrsim \omega $ we note that  \eqref{eq:2Nextbis} and \eqref{eq:expEstonsubmultipicative} imply
\begin{multline*}
\omega (x -y, \xi_1 - \eta_1, \dots, \xi_k - \eta_k) \gtrsim
\frac{\omega (x , \xi_1, \dots, \xi_k )}{v (-y, - \eta_1, \dots,  - \eta_k)}  \\[1ex]
\gtrsim \omega (x , \xi_1, \dots, \xi_k ) e^{-r(|y|+|\eta_1|+\dots +|\eta_k|)},
\end{multline*}
for some $r>0.$ Therefore,
\begin{multline*}
|\omega _0 (x, \xi_1,\dots,\xi_k)| = | \omega *\phi(x, \xi_1,\dots,\xi_k)|
\\[1ex]
\left | \iint\cdots\int
\omega (x -y, \xi_1 - \eta_1, \dots, \xi_k - \eta_k)
\phi (y, \eta_1, \dots, \eta_k)\, dy d\eta_1 \cdots d\eta_k  \right |
\\[1ex]
\gtrsim \omega (x , \xi_1, \dots, \xi_k ) \\[1ex]
\times
 \iint\cdots\int e^{-r(|y|^{\frac 1{1-s_0}}+|\eta_1|^{\frac 1{1-s_1}}+\dots +|\eta_k|^{\frac 1{1-s_k}})}
\phi (y, \eta_1, \dots, \eta_k)\, dy d\eta_1 \cdots d\eta_k
\\[1ex]
\asymp \omega  (x, \xi_1,\dots,\xi_k),
\end{multline*}
for some $r>0,$ so that $ \omega _0 \gtrsim \omega $, and (1) follows. This also gives the second part of (2).
\end{proof}

\begin{proof}[Proof of Proposition \ref{Prop:CalculiTransfer}]

It is no restriction to assume that $r_2=t_2=0$. Let $(r,t)=(r_1,t_1)$.
The proof is a straightforward application of the Fourier inversion formula,
see also \cite[Proposition 1.1]{To24}  and  \cite[Section 18.5]{Ho1} .

\par

The equality $\op _{r,t}(a)=\op (b)$ is the same as
$$
\mascF _{2,3} (b(x,\cdo,\cdot\cdot ))(y-x,z-x)
=
\mascF _{2,3} (a(x+r(y-x)+t(z-x),\cdo,\cdot\cdot ))(y-x,z-x) 	
$$
$$
\Leftrightarrow
$$
\begin{multline*}
\iint b(x,\xi_1,\eta_1 )e^{i(\scal y{\xi_1} + \scal z{\eta_1} )}\, d\xi_1 d\eta_1 
\\[1ex]
= \iint a(x+ry+tz,\xi_1,\eta_1 )e^{i(\scal y{\xi_1} + \scal z{\eta_1} )}\, d\xi_1 d\eta_1
\end{multline*}
$$
\Leftrightarrow
$$
$$ 
b(x,\xi,\eta ) = \frac{1}{(2\pi )^{2d}} \iiiint a(x+ry+tz,\xi_1,\eta_1 )
e^{i(\scal y{\xi_1-\xi} + \scal z{\eta_1-\eta} )}\, dydzd\xi_1 d\eta_1.
$$
Let $Y=(y,z)$, $Y_2=(y _2,z _2)$, $\Xi=(\xi,\eta)$
and $\Xi _1=(\xi _1,\eta _1)$.
Then by the Fourier inversion formula we get
\begin{multline*}
b(x,\xi,\eta )=
\\[1ex]
 (2\pi )^{-\frac{7d}{2}}\iint\!\!\!\iiint\widehat a(\zeta ,y_2,z_2 )
e^{i( \scal {x+ry+tz}{\zeta}+\scal{Y_2}{\Xi_1} +\scal Y{\Xi_1-\Xi}  )}
\,dz_2dy_2d\zeta   dYd\Xi_1
\\[1ex]
= (2\pi )^{-\frac{3d}{2}}
\iiint \widehat a(\zeta ,y_2,z_2 )
e^{i( \scal {x}{\zeta}-\scal{ry_2+tz_2}{\zeta}
	+\scal{y_2}{\xi}+\scal{z_2}{\eta})}
\,dz_2dy_2d\zeta
\\[1ex]
=	(2\pi )^{-\frac{3d}{2}}
\iiint e^{i(\scal{y_2}{\xi}+\scal{z_2}{\eta}+\scal x\zeta)}
\prn{e^{-i\scal{ry_2+tz_2}{\zeta}}\widehat{a}(\zeta ,y_2,z_2)}
\,dz_2dy_2d\zeta
\\[1ex]
=
e^{-i\scal{rD_\xi+tD_\eta}{D_x}}a(x,\xi,\eta ),
\end{multline*}
which gives the result.	
\end{proof}

\section*{Acknowledgements}
The authors are grateful to professor Joachim Toft for
valuable remarks and suggestions. N. Teofanov 
is partially supported by MPNTR Project No 174024.

\par


\end{document}